\documentclass[a4paper]{amsart}

\usepackage{amsmath, amssymb, amsbsy}
\usepackage{array}
\usepackage{graphpap, color, paralist, pstricks}
\usepackage[mathscr]{eucal}
\usepackage[pdftex]{graphicx}
\usepackage[pdftex,colorlinks,backref=page,citecolor=blue]{hyperref}
\usepackage{pifont}

\usepackage{mathtools}
\usepackage{cleveref}
\usepackage{float}
\usepackage{crossreftools}

\newcounter{bullet}

\usepackage{setspace}
\setdisplayskipstretch{1}

\usepackage{geometry}
\usepackage{comment}
\newtheorem{thm}{Theorem}[section]
\newtheorem{prop}[thm]{Proposition}

\newtheorem{lem}[thm]{Lemma}

\newtheorem{conj}[thm]{Conjecture}

\theoremstyle{definition}
\newtheorem{mydef}[thm]{Definition}
\newtheorem{claim}[thm]{Claim}

\newtheorem{remark}[thm]{Remark}

\crefname{lem}{lemma}{lemmas}
\crefname{thm}{theorem}{theorems}

\newcommand{\gO}{\Omega}

\newcommand{\gl}{\lambda}

\newcommand{\Z}{\mathbb{Z}}

\newcommand{\cO}{\mathcal{O} }

\newcommand{\cC}{\mathcal{C} }
\newcommand{\cD}{\mathcal{D} }
\newcommand{\cE}{\mathcal{E} }
\newcommand{\cF}{\mathcal{F} }
\newcommand{\cG}{\mathcal{G} }

\newcommand{\cI}{\mathcal{I} }

\newcommand{\cM}{\mathcal{M} }

\newcommand{\cP}{\mathcal{P} }

\newcommand{\cU}{\mathcal{U} }
\newcommand{\cV}{\mathcal{V} }
\newcommand{\cW}{\mathcal{W} }

\newcommand{\beq}[1]{\begin{equation}\label{#1}}
\newcommand{\enq}[0]{\end{equation}}

\newcommand{\eps}{\epsilon}

\newcommand{\gd}[0]{\delta }

\newcommand{\nin}[0]{\noindent}

\newcommand{\sub}[0]{\subseteq}

\newcommand{\lam}{\lambda}

\renewcommand{\eta}{\left(\left(\frac{q}{2}\right)^2\right)}

\newcommand{\qr}{\gamma}

\begin{document}

\title{A refined graph container lemma and applications to the hard-core model on bipartite expanders}

\author[M. Jenssen]{Matthew Jenssen}
\address{Department of Mathematics, King's College London}
\email{matthew.jenssen@kcl.ac.uk}
\author[A. Malekshahian]{Alexandru Malekshahian}
\address{Mathematical Institute, University of Oxford}
\email{alex.malekshahian@maths.ox.ac.uk}
\author[J. Park]{Jinyoung Park}
\address{Department of Mathematics, Courant Institute of Mathematical Sciences, New York University}
\email{jinyoungpark@nyu.edu}

\begin{abstract}
    We establish a refined version of a graph container lemma due to Galvin and discuss several applications related to the hard-core model on bipartite expander graphs.  Given a graph $G$ and $\lam>0$, the hard-core model on $G$ at activity $\lambda$ is the probability distribution $\mu_{G,\lam}$ on independent sets in $G$ given by $\mu_{G,\lam}(I)\propto \lam^{|I|}$. As one of our main applications, we show that the hard-core model at activity $\lambda$ on the hypercube $Q_d$ exhibits a `structured phase' for $\lambda= \Omega( \log^2 d/d^{1/2})$ in the following sense: in a typical sample from $\mu_{Q_d,\lam}$, most vertices are contained in one side of the bipartition of $Q_d$. This improves upon a result of Galvin which establishes the same for $\lam=\Omega(\log d/ d^{1/3})$. As another application, we establish a fully polynomial-time approximation scheme (FPTAS) for the hard-core model on a $d$-regular bipartite $\alpha$-expander, with $\alpha>0$ fixed, when $\lambda= \Omega( \log^2 d/d^{1/2})$. This improves upon the bound $\lam=\Omega(\log d/ d^{1/4})$ due to the first author, Perkins and Potukuchi. 
     We discuss similar improvements to results of Galvin-Tetali, Balogh-Garcia-Li and Kronenberg-Spinka. 
\end{abstract}

\maketitle

\section{Introduction}
Given a graph $G$, let $\mathcal{I}(G)$ denote the collection of independent sets in $G$.
 The hard-core model on $G$ at activity $\lambda>0$ is the probability distribution on $\mathcal I(G)$ given by
 \begin{align}\label{eqHCUnidef}
 \mu_{G,\lambda}(I)=\frac{\lambda^{|I|}}{Z_G(\lambda)}
 \end{align}
for $I\in \mathcal{I}(G)$, where the normalising constant
 \begin{align}\label{eqHCPartdef}
Z_G(\lambda)= \sum_{I\in \mathcal{I}(G)}\lambda^{|I|}
\end{align}
is known as the hard-core model partition function. 
The hard-core model originated in statistical physics as a simple model of a gas. 
The vertices of the graph $G$ are to be thought of as `sites' that can be occupied by particles, and neighbouring sites cannot both be occupied.  
This constraint models a system of particles with `hard cores' that cannot overlap. In statistical physics, a major motivation for studying the hard-core model is that it provides a setting where the notion of \emph{phase transition} can be rigorously investigated. In this context, the most common host graph of study is (the nearest neighbour graph on) the integer lattice $\mathbb Z^d$ (see Section~\ref{conclusion} for a more precise discussion of phase transitions and hard-core measures on $\mathbb Z^d$). For now a phase transition can be loosely thought of as follows: as $\lam$ increases, a typical sample from the hard-core model on (a large box in) $\mathbb Z^d$ transitions from being disordered to being structured, in the sense that it prefers vertices from either the odd or even sublattice. 
Motivated by this phenomenon, Kahn~\cite{kahn2001entropy} initiated the study of the hard-core model on the hypercube $Q_d$ (and regular bipartite graphs in general). Here $Q_d$ denotes the graph on vertex set $\{0,1\}^d$ where two vertices are adjacent if and only if they have Hamming distance $1$. $Q_d$ is a bipartite graph with bipartition $\cE \cup \cO$, where $\cE$ {and} $\cO$ consist of the vertices with even {and} odd Hamming weight, respectively. Kahn showed that for fixed $\lam, \eps>0$ and $\mathbf I$ sampled according to $\mu_{Q_d,\lam}$ both
\[
\left||\mathbf I| - \frac{\lam}{1+\lam}2^{d-1} \right|\leq \frac{2^d}{d^{1-\eps}}
\]
and 
\[
\min\{|\mathbf I \cap \cE|, |\mathbf I \cap \cO |\}\leq  \frac{2^d}{d^{1/2-\eps}}
\]
hold whp (that is, with probability tending to $1$ as $d\to\infty$). Roughly speaking, these results show that for $\lam>0$ fixed, a sample from $\mu_{Q_d,\lam}$ resembles a random subset of either $\cE$ or $\cO$ where each element is chosen independently with probability $\lam/(1+\lam)$. In other words, samples from $\mu_{Q_d,\lam}$ exhibit a significant degree of structure. This lies in stark contrast to the regime $\lam\leq c/d$ ($c$ small) where a sample $\mathbf I$ from $\mu_{Q_d,\lam}$ resembles a $\lam/(1+\lam)$-random subset of $\cE \cup \cO$; in particular, $|\mathbf I \cap \cE|=(1+o(1))|\mathbf I \cap \cO|$ whp -- see \cite{weitz}.

Galvin~\cite{galvin_threshold} later refined Kahn's results, showing that the structured regime holds all the way down to $\lam=\tilde \Omega(d^{-1/3})$. More precisely, he showed that there exists $C>0$ so that if $ C \log d/d^{1/3}\leq \lam \leq \sqrt{2}-1$ and $\mathbf I$ is a sample from $\mu_{Q_d,\lam}$, then whp
\[
\left||\mathbf I| - \frac{\lam}{1+\lam}2^{d-1} \right|\leq d \log d \left(\frac{2}{1+\lam} \right)^d
\]
and 
\[
\frac{1}{4 \log m}\frac{\lam}{2}\left(\frac{2}{1+\lam} \right)^d\leq \min\{|\mathbf I \cap \cE|, |\mathbf I \cap \cO |\}\leq em^2\frac{\lam}{2}\left(\frac{2}{1+\lam} \right)^d\, 
\]
for some $m=m(\lam,d)=o(d/\sqrt{\log d})$. Similar bounds hold when $\lam>\sqrt{2}-1$, but they take a slightly different form (see~\cite[Theorem 1.1]{galvin_threshold}). 

 Galvin's proof is based on Sapozhenko's graph container method~\cite{sapocontainer} (see Section~\ref{subsec:containers} for more on the container method). This influential method has now enjoyed numerous applications in the combinatorics literature. The restriction that $\lam=\tilde \Omega(d^{-1/3})$ in Galvin's result is an artifact of a graph container lemma and as a consequence, the same restriction appears in many other applications of similar lemmas~\cite{balogh2021independent,galvin2006slow,struct,jenssen2023approximately,kronenberg2022independent}.

 Galvin's results~\cite{galvin_threshold} were extended by the first author and Perkins~\cite{hypercube}, who combined the graph container method with a method based on the theory of polymer models and cluster expansion from statistical physics. This allows for a very precise description of the hard-core measure on $Q_d$ as a `perturbation' of the measure which selects a side $\cE,\cO$ uniformly at random and then selects a $\lam/(1+\lam)$-subset of that side. As a consequence one can obtain detailed asymptotics for the partition function $Z_{Q_d}(\lam)$ and determine the asymptotic distribution of $|\mathbf I|$ and $\min\{|\mathbf I \cap \cE|, |\mathbf I \cap \cO |\}$. These results rely on Galvin's container lemma and are therefore also limited to the regime $\lam=\tilde\Omega(d^{-1/3})$.

Here we prove a refined graph container lemma (\Cref{ML} below), which allows us to extend these structure theorems to the range $\lam=\tilde \Omega(d^{-1/2})$. One can make similar improvements to several other applications of the graph container method and we discuss these further applications in \Cref{subsec:furtherapp}. 

\begin{thm}\label{thm:main} There exists $C>0$ so that the following holds.  Let $C\log^2d /d^{1/2}\leq \lam \leq \sqrt{2}-1$ and let $\mathbf{I}$ be sampled according to $\mu_{Q_d,\lam}$. Then, with high probability,
 \beq{expbig} \left| |\mathbf{I}| - \frac{\lam}{1+\lam}2^{d-1}\right| \leq \omega(1)\cdot\frac{d\lam^2 2^d}{(1+\lam)^d} \enq
 and
    \beq{expcube} \min\{ |\mathbf{I}\cap \cE|, |\mathbf{I}\cap \cO|\} =(1+o(1))\frac{\lam}{2}\cdot\left(\frac{2}{1+\lam}\right)^d\enq
    where $\omega(1)$ is any function tending to infinity as $d\rightarrow\infty$.
\end{thm}

We reiterate that the regime $\lam\geq C \log d/d^{1/3}$ was already treated in detail in~\cite{galvin_threshold,hypercube} and the regime $\lam > \sqrt{2}-1$ was treated in~\cite{galvin_threshold}. \Cref{thm:main} shows that the `structured regime', where a typical independent set is imbalanced, persists all the way down to $\lam=\tilde \Omega(d^{-1/2})$. Theorem~\ref{thm:main} is an easy consequence of a detailed description of $\mu_{Q_d,\lam}$ (see Theorem~\ref{clusterapprox} below) analogous to the one obtained in~\cite{hypercube}.  This structure theorem has several other consequences (many of which are elaborated upon in~\cite{hypercube}), but here we focus on Theorem~\ref{thm:main} as our main application for brevity. 

We describe our new graph container lemma (\Cref{ML}) in the following section, which is our main technical contribution. We then go on to discuss further applications of the lemma. 

\subsection{An improved graph container lemma}\label{subsec:containers}

The container method is a classical tool that has seen widespread use in the context of studying independent sets in graphs. Its roots can be traced back to the work of Kleitman and Winston \cite{kleitman1982number} and Sapozhenko \cite{sapocontainer}. In recent years, the method has been generalized and developed into a powerful approach for studying independent sets in hypergraphs in the celebrated work of Balogh, Morris and Samotij \cite{balogh2015independent} and Saxton and Thomason \cite{saxton2015hypergraph}. This method has enjoyed a wealth of applications in extremal, enumerative and probabilistic combinatorics and beyond.

In this paper, we are interested in the graph container method specialized to the case of bipartite graphs $\Sigma=X \cup 
 Y$. A first result of this type was established by Sapozhenko \cite{sapocontainer}. Sapozhenko's method was elaborated upon by Galvin~\cite{Galvin_independent} and the form of our container lemma is closely modelled after his results (see also \cite{galvin_threshold, galvin2006slow}).  These results have seen numerous applications which we discuss further in~\Cref{subsec:furtherapp}.

We formulate our new container lemma in enough generality to encompass all of the applications outlined in~\Cref{subsec:furtherapp} (not just Theorem~\ref{thm:main}) and with a view to future applications. The statement is somewhat technical and so we first set up some notation.

For $\delta\geq 1$ and $d_Y \le d_X$, we say a bipartite graph $\Sigma=X \sqcup Y$ is \emph{$\delta$-approximately $(d_X, d_Y)$-biregular} if it satisfies 
\[d(v) \in \begin{cases}
    [d_X, \delta d_X] &\forall v\in X, \\
    [\delta^{-1}d_Y, d_Y]  & \forall v\in Y
    \end{cases}\]
where $d(v)$ denotes the degree of $v$.
For the rest of the paper, we assume $\Sigma$ is $\gd$-approximately $(d_X, d_Y)$-biregular. 

For $A \sub X$, the \textit{closure} of $A$ is defined to be  $[A]:=\{x \in X: N(x) \sub N(A)\}.$ A set $S \sub V(\Sigma)$ is \textit{2-linked} if $\Sigma^2[S]$ is connected, where $\Sigma^2$ is the square of $\Sigma$ (i.e. $V(\Sigma^2)=V(\Sigma)$ and two vertices $x,y$ are adjacent in $\Sigma^2$ iff their distance in $\Sigma$ is at most 2). Given $a, g\in\mathbb{N}$, define
\begin{align}\label{eq:GagDef}
\cG(a,g)=\cG(a,g,\Sigma)=\{A \subseteq X \text{ 2-linked }: |[A]|=a \text{ and } |N(A)|=g\},
\end{align}
and set $t\coloneqq g-a$ and $w\coloneqq gd_Y-ad_X$.
The technical-looking definition below is a measurement of the expansion of $\Sigma$.  Given $1\leq \varphi \leq \delta^{-1}d_Y-1$, let
\[m_\varphi=m_\varphi(\Sigma)=\min\{|N(K)|: y \in Y, K \sub N(y), |K| >\varphi\}.\]
The proof of \Cref{ML} crucially relies on the expansion of $\Sigma$, and the lower bounds on $m_\varphi$ and $g-a$ below provide quantification of the expansion that we need. All the graphs that we consider in our applications will satisfy these lower bounds (except in \Cref{algos}, where we require a stronger bound on $g-a$ and we relax the assumption on $m_{\varphi}$).

\begin{lem}\label{ML} Let $d_Y \le d_X$ be sufficiently large integers and let $\gd \ge 1$, $\gd',\gd''>0$. Then there exist $c=c(\gd, \gd',\gd'')>0$ and $C=C(\gd, \gd',\gd'')>0 $ such that the following holds. Let $\Sigma=X\sqcup Y$ be $\delta$-approximately $(d_X, d_Y)$-biregular  such that ${m_\varphi} \ge \delta''\cdot (\varphi d_X)$, where $\varphi=d_Y/(2\gd)$. If $a,g \in \mathbb N$ satisfy $g-a \ge \max\left\{\frac{\delta' g}{d_Y}, \frac{cd_X}{(\log d_X)^2}\right\}$ and $\gl>\frac{C\log^2 d_X}{(d_X)^{1/2}}$, then
\beq{target}\sum_{A \in \cG(a,g)} \gl^{|A|} \le |Y|(1+\gl)^g e^{-(g-a)\log^2 d_X/(6d_X)}.\enq
\end{lem}

We note that the main contribution of the above lemma is the improved lower bound on $\gl$, which was previously (essentially\footnote{Galvin considers regular graphs only, so that $d_X=d_Y=d$.}) $\tilde\gO(d_X^{-1/3})$ in \cite{galvin_threshold}. As is typical with existing results of this type, the proof of \Cref{ML} consists of two parts: an algorithmic procedure for constructing graph containers efficiently, and a `reconstruction' argument that allows us to bound the sum on the left hand-side of \eqref{target} given the family of containers we have constructed. The main driving force behind our improvement in the range of $\lam$ is a novel approach to the {container} construction algorithm, \Cref{heart}.

We conjecture that a bound of the type given in \Cref{ML} should hold for $\lam=\Tilde{\Omega}(1/d_X)$.
\begin{conj}\label{conj}
    There exists a constant $\kappa$ and a function $\lam^\ast:\mathbb N\to\mathbb R$ with $\lam^\ast(d)=\Tilde{\Omega}(1/d)$ as $d\to\infty$ such that, under the assumptions of \Cref{ML}, if $\lam\geq\lam^\ast(d_X)$, then
    \[ \sum_{A\in\cG(a, g)}\lam^{|A|} \leq |Y| (1+\lam)^g \exp \left\{ -(g-a)/d_X^{\kappa}\right\}.\]
\end{conj}

\begin{remark}
We note that the assumption $g-a =\Omega(g/d_Y)$ in \Cref{ML} can be relaxed to $g-a= \Omega(\frac{g\log^2 d_X}{\sqrt{d_X} d_Y \gl})$ (which is weaker for $\gl \gg \frac{\log^2 d_X}{(d_X)^{1/2}}$) as long as $\gl$ grows at most polynomially fast in $d_X$. We stick to the simpler lower bound $g/d_Y$ for simplicity, as this lower bound is enough for all of the current applications.
\end{remark}

\subsection{Further applications}\label{subsec:furtherapp}

\subsubsection*{Approximation algorithms for the hard-core model on expanders}
Given $Z,\hat Z, \epsilon>0$, we say that $\hat Z$ is an $\epsilon$-relative approximation to $Z$ if $ e^{-\eps} Z \leq \hat Z \leq  e^{\eps} Z$. Recall that for two probability measures $\mu, \nu$ on a measurable space $(\Omega,\mathcal{F})$, the total variation distance between $\mu$ and $\nu$ is
\[
\|\mu- \nu\|_{TV}= \sup_{A\in \mathcal{F}}|\mu(A)-\nu(A)|\, .
\]
Two natural computational tasks arise when considering the hard-core model on $G$ at activity $\lam$:
\begin{enumerate}
\item  Compute an $\epsilon$-relative approximation to $Z_G(\lam)$; \label{FPTAS}
\item Output an independent set with distribution $\hat \mu_{G,\lam}$ such that $\|\hat \mu_{G,\lam}- \mu_{G,\lam}\|_{TV}\leq \eps$.\label{sample}
\end{enumerate}
A deterministic algorithm which does Task~\ref{FPTAS} in time polynomial in $n$ and $1/\eps$ is known as a \emph{fully polynomial time approximation scheme} (FPTAS).
An algorithm which does Task~\ref{sample} is known as an \emph{efficient sampling scheme}.

Intuitively, one might expect the problem of approximating $Z_G(\lam)$ to be easier on the class of bipartite graphs; for one, there is a polynomial-time algorithm to find a maximum-size independent set in  a bipartite graph while the corresponding problem is NP-hard for general graphs. The problem of approximating $Z_G(\lam)$ for bipartite $G$ belongs to the complexity class  \#BIS introduced by Dyer, Goldberg, Greenhill, and Jerrum~\cite{dyer2000relative}. They showed that several natural combinatorial counting problems are as hard to approximate as \#BIS. Resolving the complexity of \#BIS remains a major open problem and in recent years there has been an effort to design approximation algorithms for $Z_G(\lam)$ that exploit bipartite structure. The line of work most relevant here is that which followed the breakthrough of Helmuth, Perkins and Regts~\cite{HelmuthAlgorithmic2} who designed efficient approximation algorithms for  $Z_G(\lam)$ on $\Z^d$ in the previously intractable `low temperature' (i.e.\ large $\lam$) regime. Their method was based on tools from statistical physics, namely Pirogov-Sinai Theory, polymer models and cluster expansions. Soon after the first author, Keevash and Perkins~\cite{JKP2} gave an FPTAS and efficient sampling scheme for the low-temperature hard-core model in bounded-degree, bipartite expander graphs. This work was followed by several improvements, extensions, and generalizations including~\cite{cannon2020counting,chen2021fast,chen2021sampling,friedrich2020polymer,galanis2020fast,galanis2021fast,liao2019counting}.  Many of these algorithms exploit the fact that on a bipartite graph $G$ with sufficient expansion, typical samples from $\mu_{G,\lam}$ are imbalanced (preferring one side of the bipartition to the other). The container method is a particularly powerful tool for detecting such structure. This fact was exploited by the first author, Perkins and Potukuchi~\cite{jenssen2023approximately} who combined the cluster expansion method with graph containers to extend the range of $\lam$ for which efficient approximation algorithms on bipartite expanders were known to exist. Our new container lemma (\Cref{ML}) can be used to extend the range further still. Given $\alpha>0$, we say that a bipartite graph with parts $X,Y$ is a \emph{bipartite $\alpha$-expander} if $|N(A)|\geq (1+\alpha)|A|$ for all $A\subseteq X$ with $|A|<|X|/2$ and $A\subseteq Y$ with $|A|<|Y|/2$. 

\begin{thm}\label{algos}
    For every $\alpha>0$ there exists a constant $C>0$ such that for all $d\geq 3$ and $\lam\geq \frac{C\log^2 d}{d^{1/2}}$ there is an FPTAS for $Z_G(\lam)$ and a polynomial-time sampling scheme for $\mu_{G, \lam}$ for the class of $d$-regular, bipartite $\alpha$-expanders.
\end{thm}

\nin The above result extends \cite[Theorem 2]{jenssen2023approximately}, which assumes that $\lam=\tilde \Omega(d^{-1/4})$. 

In the remainder of this section, we discuss further applications more briefly and informally {for the sake of brevity and in order to avoid excessive} repetition of previous work. For these applications we do not provide formal proofs, and only indicate where improvements can be made.

\subsubsection*{Slow mixing of Glauber dynamics.} Given a graph $G$, the \emph{Glauber dynamics} for the hardcore model $\mu_{G,\lam}$ is the following Markov chain on state space $\cI(G)$, the family of independent sets of $G$:
\begin{enumerate}
\item Begin with an arbitrary $I\in\cI(G)$, e.g. $I=\emptyset$.
\item Choose a vertex $v\in V(G)$ uniformly at random.
\item {Sample} $X\sim \text{Ber}(\lam/(1+\lam))$. If $X=1$ and $v$ has no neighbours in $I$ then update $I\leftarrow I\cup \{v\}$. If $X=0$ update $I\leftarrow I\backslash \{v\}$.
\end{enumerate}
This Markov chain has stationary distribution $\mu_{G,\lam}$. If the chain mixes rapidly, it gives an efficient sampling scheme for the hard-core model on $G$. Galvin and Tetali in \cite{galvin2006slow} investigated the mixing time $\tau_{\cM_{\lam}(G)}$ of the Glauber dynamics for $\mu_{G,\lam}$ on a bipartite expander $G$. They showed that for $\lam$ sufficiently large the Glauber dynamics mixes slowly. The driving force behind this slow mixing is that, as alluded to in previous sections, a typical sample from $\mu_{G,\lam}$ is highly imbalanced. Balanced independent sets therefore create a small `bottleneck' in the state space. As a concrete example, Galvin and Tetali prove the following. 

\begin{thm}[{\cite[Corollary 1.4]{galvin2006slow}}] 
    There exists $C>0$ such that if $d$ is sufficiently large and $\frac{C\log^{3/2} d}{d^{1/4}}\leq \lam\leq O(1)$, we have
    \[ \tau_{\cM_\lam (Q_d)} \geq \exp\left\{\Omega\left(\frac{2^d \log^3(1+\lam)}{\sqrt{d}\log^2 d} \right)\right\}.\]
\end{thm}

Our improved container lemma (\Cref{ML}) used in place of \cite[Theorem 2.1]{galvin2006slow}
establishes the same slow mixing phenomenon all the way down to $\lam=\tilde \Omega(d^{-1/2})$. More generally, Galvin and Tetali establish slow mixing results for $d$-regular bipartite $\alpha$-expanders provided $\lam$ is sufficiently large as a function of $\alpha$ and $d$. As above, \Cref{ML} can be used to improve the range of $\lam$ for which these results hold.

\subsubsection*{Antichains in the Boolean lattice.} Given $n\in\mathbb N$, let $B_n$ denote the Boolean lattice of dimension $n$, i.e. the power set of $[n]=\{1,\ldots,n\}$.
Recall that $\cF \subseteq B_n$ is called an \emph{antichain} if it is inclusion-free (i.e. $A \not \subseteq B$ for all distinct $A,B\in \cF$). Dedekind's problem, dating back to 1897, asks for the total number of antichains in $B_n$. In recent work~\cite{sparsesperner}, the current authors applied the graph container method to study Dedekind's problem in detail. As part of this study, the authors count and study the typical structure of antichains in $B_n$ of a \emph{given} size. For $0\leq k\leq n$, we call the family $L_k=\binom{[n]}{k}\subseteq B_n$ the $k^{th}$ $\emph{layer}$ of $B_n$. We say that $L_{k-1},L_{k}, L_{k+1}$ are three \emph{central layers} if $k=\lfloor n/2 \rfloor$ or $k= \lceil n/2 \rceil$.
One of the {main} results of \cite{sparsesperner} is the following.

\begin{thm}[{\cite[Theorem 1.5]{sparsesperner}}]\label{thm:antichain}
    There exists $C>0$ such that if  $\frac{C\log^2 n}{\sqrt{n}}<\beta\leq 1$, then almost all antichains of size $\beta \displaystyle \binom{n}{\lfloor n/2\rfloor}$ in $B_n$ are a subset of three central layers.
\end{thm}

For the proof, we study the hard-core model on the graph $G_n$ whose vertex set is $B_n$ and $u\sim v$ if and only if {$u\neq v$ and} $u$ is contained in $v$ or vice versa. Note that independent sets in $G_n$ are precisely antichains in $B_n$. We use our container lemma, \Cref{ML} (and its proof), to study the hard-core measure on $G_n$ and to prove Theorem~\ref{thm:antichain}.

Recently Balogh, Garcia and Li \cite{balogh2021independent} studied the hard-core model on the subgraph $M_n\subseteq G_n$ induced by the middle two layers $L_{(n-1)/2}, L_{(n+1)/2}$ when $n$ is odd. In particular, they obtain detailed asymptotics for the partition function $Z_{M_n}(\lam)$ when $\lam\geq C \log n/n^{1/3}$ (see \cite[Theorem 5.1]{balogh2021independent}). The now familiar $\tilde\Omega(n^{-1/3})$ bottleneck comes from the graph container method. Using \Cref{ML} one can easily extend \cite[Theorem 5.1]{balogh2021independent} to the regime $\lam= \tilde\Omega(n^{-1/2})$.

\subsubsection*{Independent sets in the percolated hypercube.}

Given $d\in \mathbb{N}$ and $p\in [0,1]$ we let $Q_{d,p}$ denote the random subgraph of the hypercube $Q_d$ where each edge is retained independently with probability $p$. Recently, Kronenberg and Spinka~\cite{kronenberg2022independent} studied the hard-core model on $Q_{d,p}$. In particular, they find asymptotic formulae for the expected value (and higher moments) of the partition function $Z_{Q_{d,p}}(\lam)$. For example, they prove the following.

\begin{thm}[{\cite[Theorem 1.1]{kronenberg2022independent}}]\label{thm:KS}
	For $d$ sufficiently large and $p \ge \frac{C\log d}{d^{1/3}}$,
	\[ \mathbb E |\cI(Q_{d,p})| =
	2 \cdot 2^{2^{d-1}} \exp\left[ \tfrac12 (2-p)^d + \left(a(p)\tbinom d2 -\tfrac14 \right) 2^d (1-\tfrac p2)^{2d} + O\left(d^4 2^d (1- \tfrac p2)^{3d}\right) \right], \]
	where
	\[ a(p) := \frac{(1+(1-p)^2)^2 }{(2-p)^4} - \frac14 .\]
\end{thm}
The restriction on $p$ here arises from an application of the graph container method. Our refined container method can be used to extend the above theorem to the range $p=\tilde\Omega(d^{-1/2})$. Kronenberg and Spinka prove several results about the hard-core model on $Q_{d,p}$ at activity $\lam$. The range of $\lam$ and $p$ for which these results hold may similarly be extended using the refined container methods of this paper. 

\subsection{Organization}  We collect preliminary results in \Cref{prelims} and prove our core result, \Cref{heart}, in \Cref{cont.proof}. We then give the reconstruction argument that enables us to deduce \Cref{ML} in \Cref{mlproofsec}. We prove \Cref{thm:main} in \Cref{pfthmmain} and \Cref{algos} in \Cref{pfthmalgos}. Finally, we conclude with some brief remarks in \Cref{conclusion}. 

\subsection{Notation and usage} We use $\Sigma=X\sqcup Y$ for a bipartite graph with parts $X$ and $Y$, and  $V(\Sigma)$ and $ E(\Sigma)$ for the set of vertices and edges of $\Sigma$, respectively. As usual, we write $N(x)$ for $\{y \in V(\Sigma): \{x,y\} \in E(\Sigma)\}$ and $d_A(x):=|N(x) \cap A|$. We use $d(x)$ for $d_\Sigma(x)$ for simplicity. For $S \sub V(\Sigma)$, set $N(S):= \cup_{x \in S} N(x)$ and let $\Sigma[S]$ denote the induced subgraph of $\Sigma$ on $S$. For $A, B \sub V(\Sigma)$, let $\nabla(A,B)=\{\{x,y\} \in E(\Sigma):x \in A, y \in B\},$ and $\nabla(A)=\nabla(A, V(\Sigma) \setminus A)$.

We write $\binom{n}{\leq k}$ for $\sum_{i \le k}\binom{n}{i}.$ We will make frequent use of the basic binomial estimate 
\beq{binom} {n \choose \le k} \le \exp\left\{k \log \left(\frac{en}{k}\right)\right\} \quad \text{ for } k \le n, \enq

\nin where here and throughout the paper $\log$ is used for the natural logarithm.

For two functions $f,g:\mathbb{N}\to \mathbb{R}$ we write $f(d)= \tilde O(g(d))$ if there exists $C\in \mathbb{R}$ such that $|f(d)|\leq (\log d)^C |g(d)|$ for $d$ sufficiently large. We use $\tilde \Omega$ analogously. Throughout the paper $\Sigma$ will denote a $\gd$-approximately $(d_X, d_Y)$-biregular graph (for some choice of  $d_Y\leq d_X$ and absolute constant $\delta$) and all asymptotic notation is to be understood with respect to the limits $d_X,d_Y\to\infty$. 
 
The \emph{cost} of a choice means the logarithm (in base two) of the number of its possibilities.  We employ a common abuse of notation by often omitting floor and ceiling symbols for notational convenience.
\newline

\section{Preliminaries for the proof of Lemma \ref{ML}}\label{prelims}

\subsection{Proof outline}\label{sec.sketch} We first provide a brief outline of the key ideas in the proof of \Cref{ML}. Our proof follows the outline of the original proof of the graph container result of Sapozhenko \cite{sapocontainer}, with a couple of key improvements that we outline below. We refer the reader to \cite{Galvin_independent} for more details of the original graph container lemma.

 The original proof of Sapozhenko's result for counting independent sets roughly consists of two phases, with the first phase splitting into two subphases:
\begin{enumerate}[I.]
\item Constructing containers
\begin{enumerate}[(i)]
    \item algorithmically produce `coarse' containers called $\varphi$-approximations (\Cref{lem.varphi}),
    \item algorithmically refine the coarse containers from the previous step into $\psi$-approximations (\Cref{heart}).
\end{enumerate}
\item Reconstructing independent sets from the containers produced by the $\psi$-approximating algorithm.
\end{enumerate}

Galvin's proof \cite{galvin_threshold} of the weighted graph container lemma follows the above structure of Sapozhenko's proof. Our proof also follows it at a high level, but the key to our improvement lies in improving the $\psi$-approximation algorithm. In the implementation of $\psi$-approximation, the choice of the value for $\psi$ has two implications: the cost and the fineness of the containers. A smaller value of $\psi$ will produce finer containers at higher cost. Then in the reconstruction phase, it is cheaper to reconstruct independent sets from finer containers. So there is a trade-off, and the lower bound on $\lambda=\tilde \Omega(d_X^{-1/3})$ in \cite{galvin_threshold} is obtained by optimizing all of those costs.

An ingredient that is new in our proof is that we iterate the $\psi$-approximating algorithm, with $\psi_i$ decreasing at every step. As mentioned above, decreasing the value of $\psi_i$ at each iteration increases the fineness of containers (and thus reduces their reconstruction cost). Here the key observation is that, for certain containers produced by the $\psi$-approximating algorithm, we can rerun a $\psi'$-approximating algorithm for some $\psi'<\psi$ without paying very much for this additional run. This means that (under certain circumstances) we can obtain finer containers for a cost much lower than a single implementation of $\psi$-approximation. \Cref{heart} in the next subsection, whose proof is deferred to \Cref{cont.proof}, provides a quantitative statement for the (reduced) cost of the iterative runs of $\psi$-approximating algorithm. 

Unfortunately, there are some `bad' containers for which the additional run of $\psi'$-approximation (as above) is as expensive as a single run. So in our iteration, we will have to choose the right values $\psi_i$ so that we balance the cost for iterative implementation of $\psi_i$-approximations against the reconstruction cost for the bad containers. The corresponding computation will be presented in  \Cref{mlproofsec}.

We note that the idea of iteratively applying a container algorithm appears in several applications of the (hypergraph) container method in the literature. We refer the reader to the excellent survey~\cite{balogh2018method} on the hypergraph container method for some further examples of this idea in action. 

\subsection{Preliminary results}
We first collect necessary definitions and claims for the proof of \Cref{ML}, and provide the statement of \Cref{heart} at the end as promised. We follow the notation of Galvin \cite{Galvin_independent}.

For $A \sub X$ and $\varphi\geq 0$ define $N(A)^\varphi:=\{y \in N(A):d_{[A]}(y)>\varphi\}$. Define a \textit{$\varphi$-approximation} for $A \sub X$ to be an $F \sub Y$ satisfying 
    \[N(A)^\varphi \sub F \sub N(A) \text{ and } N(F) \supseteq [A].\]
Below is a version of Sapozhenko's $\varphi$-approximation lemma which is extended for our approximate biregular setting. Recall that $\Sigma=X \sqcup Y$ is $\gd$-approximately $(d_X, d_Y)$-biregular, $\cG(a,g)=\{A \subseteq X \text{ 2-linked }: |[A]|=a \text{ and } |N(A)|=g\}$, 
$t=g-a$, $w=gd_Y-ad_X$, and $m_\varphi=\min\{|N(K)|: y \in Y, K \sub N(y), |K| >\varphi\}.$ In the remainder of this paper, we also assume that $d_X$ and $d_Y$ are sufficiently large for all of our computations to go through. 

We note that the degree assumption $d_Y \le d_X$ is crucially used in \Cref{SF_size} (see the sentences following \Cref{SF_size}). On the other hand, the proofs of \Cref{lem.varphi} and \Cref{heart} go through without the assumption $d_Y \le d_X$, however, the conclusions would need to be modified slightly depending on the relative size of $d_Y$ and $d_X$.

\begin{lem}[$\varphi$-approximation for approximately biregular graphs]\label{lem.varphi} Fix $1 \le \varphi \le \frac{d_Y}{\gd}-1$, and let $\gamma'>0$ satisfy $\gamma'\log d_X/(\varphi d_X)<1.$ 
Then there is a family $\cV=\cV(a,g,\varphi) \sub 2^Y$ with
\beq{varphi_bd}|\cV|\le |Y|\exp\left\{\frac{54g\gamma'\log d_X \log (\delta d_Xd_Y)}{\varphi d_X}+\frac{54g\log (\delta d_Xd_Y)}{(d_X)^{\gamma'm_\varphi/(\varphi d_X)}}+\frac{54w\log d_Y \log (\delta d_Xd_Y)}{d_X(d_Y/\gd-\varphi)}\right\}{\frac{3g\gamma'd_Y\log d_X}{\varphi d_X}\choose \le \frac{3w\gamma'\log d_X}{\varphi d_X}}\enq
such that each $A \in \cG(a,g)$ has a $\varphi$-approximation in $\cV.$    
\end{lem}

\nin The proof of \Cref{lem.varphi} follows almost identically to that of \cite[Lemma A.3]{BalGLW21}, with the exception that we replace the biregular assumption by approximate biregularity. For completeness, we give a proof in Appendix~\ref{app:varphi}. 

The following definition originates from the notion of $\psi$-approximation of Sapozhenko. We relax this definition to allow the separate cutoffs $\psi_X$ and $\psi_Y$, and the values of $(\psi_X, \psi_Y)_i$ will evolve in the iteration of $\psi_i$-approximations.

\begin{mydef}\label{def:psi-approximation}
    For $1 \le \psi_X\leq d_X-1$ and $1\leq \psi_Y \le \frac{d_Y}{\delta}-1$, a \textit{$(\psi_X, \psi_Y)$-approximating pair} for $A \sub X$ is a pair $(F,S) \in 2^Y \times 2^X$ satisfying $F \sub N(A)$, $S \supseteq [A]$,
    \beq{S.degree} d_F(u) \ge d(u)-\psi_X \quad \forall u \in S, \text{ and}\enq
    \beq{F.degree} d_{X \setminus S}(v) \ge d(v)-\psi_Y \quad \forall v \in Y \setminus F.\enq
\end{mydef}

\begin{prop}\label{SF_size}
    If $(F,S)$ is a $(\psi_X, \psi_Y)$-approximating pair for $A \in \cG(a,g)$, then, with $s:=|S|$ and $f:=|F|,$
    \beq{eq.SF} s \le f+\frac{1}{d_X}[(g-f)\psi_Y+(s-a)\psi_X].\enq
\end{prop}

\begin{proof}
Observe that, with $\bar F:=Y \setminus F$,
\[|\nabla(S, \bar F)|=|\nabla([A], N(A) \setminus F)|+|\nabla(S \setminus [A], \bar F)| \stackrel{\eqref{S.degree},\eqref{F.degree}}{\le} (g-f)\psi_Y+(s-a)\psi_X.\]
Therefore,
\[s d_X\leq|\nabla S|=|\nabla(S,F)|+|\nabla(S, \bar F)| \le d_Y f+[(g-f)\psi_Y+(s-a)\psi_X],\]
which yields the conclusion by recalling that $d_Y\leq d_X$.
\end{proof}

We highlight that Proposition~\ref{SF_size} makes crucial use of the assumption $d_Y\leq d_X$. In our applications, Proposition~\ref{SF_size} will allow us to conclude that $s \le (1+o(1))f$ which is key to the proof of \Cref{ML}.

For future reference, we record a very loose upper bound on $|S|$ that follows from the above proposition: by applying the bounds $\psi_X, \psi_Y\leq d_X-1$ to \eqref{eq.SF}, we have $s \le g-(g-f)/d_X+(s-a)(1-1/d_X)\le g+(1-1/d_X)s,$ from which it follows that
\beq{s_loose} s \le d_Xg. \enq

\Cref{heart} below is at the heart of the proof of \Cref{ML}. As we sketched in \Cref{sec.sketch}, the proof of \Cref{heart} is based on the idea of the iterative run of $\psi_i$-approximating algorithms. For the parameter $i$ in the statement below, it can be understood that the corresponding conclusions are derived from the $i$-th run.

\begin{lem}[Iterative $\psi$-approximation]\label{heart} 
For any sufficiently large absolute constant $\gamma$, the following holds. Fix $1\leq \varphi\leq \frac{d_Y}{\gd}-1$. Fix  $F'\subseteq Y$ and write $\cG(a,g,F')$ for the collection of $A \in \cG(a,g)$ for which $F'$ is a $\varphi$-approximation. Then there exists a family $\mathcal{F}=\mathcal{F}(a, g, F')\subseteq 2^Y \times 2^X$ that satisfies the following properties: with
\beq{kap} \kappa:=\max\{i:2^i \le \sqrt{d_X}\},\enq
we have that
\begin{enumerate}
    \item $\cF=\bigcup_{i \in [0, \kappa]} \cF_i \left(=\bigcup_{i \in [0,\kappa]}\cF_i(a,g,F')\right)$ where for each $i \in [0,\kappa]$,
    \[|\cF_i|\le {d_Yg \choose \le \frac{\delta \gamma w}{(d_Y-\delta \varphi)d_X}}{\delta d_X d_Y g \choose \le \frac{2\gamma w}{d_Xd_Y}} \exp\left[ O\left(\frac{it\gamma \log(gd_X^2/t)}{d_X}\right)\right];\]
     \item if $(F,S)\in\cF_i$, then $(F,S)$ is a $(d_X/(2^i\gamma), d_Y/\gamma)$-approximating pair for some $A \in \cG(a,g,F')$; in addition, for $i \in [0, \kappa-1]$, every $(F, S) \in \cF_i$ satisfies $|F|<g-t/2^i$; and
    \item every $A \in \cG(a,g,F')$ has a  $(d_X/(2^i\gamma), d_Y/\gamma)$-approximating pair in $\cF_i$ for some $i \in [0, \kappa]$.
\end{enumerate}
\end{lem}

Noting that $\kappa \le \log d_X$ and $t\log^2 d_X/d_X$ is bounded away from $0$ (by the second lower bound on $t$ in \Cref{ML}), it follows that for any given $a,g$ and any $F' \in \cV(a,g,\varphi)$,
\beq{heart.bd}\sum_{i \in [0, \kappa]}|\cF_i(a,g,F')|\le {d_Yg \choose \le \frac{\delta \gamma w}{(d_Y-\delta \varphi)d_X}}{\delta d_X d_Y g \choose \le \frac{2\gamma w}{d_Xd_Y}} \exp\left[ O\left(\frac{t\gamma\log d_X \log(gd_X^2/t)}{d_X}\right)\right].\enq

\section{Proof of \Cref{heart}}\label{cont.proof}

Again, the proof of \Cref{heart} is based on an iterative $\psi$-approximating algorithm. Lemmas \ref{lem.psi} and \ref{lem.heart} below bound the cost for the initial and an additional run of the $\psi$-approximating algorithm, respectively. Crucially, we obtain the bound in \eqref{psi'_bd} which will be significantly smaller than that in \eqref{psi_bd} in our setting, which is a key to our improvement.

\begin{lem}[Cost for initial run]\label{lem.psi} For any $F' \in \cV(a,g,\varphi)$ and $1 \le \psi_X \leq d_X-1, \ 1\leq \psi_Y \le d_Y/\delta-1$ there is a family $\cW=\cW(F', \psi_X, \psi_Y) \sub 2^Y \times 2^X$ with
\beq{psi_bd}|\cW|\le {d_Yg \choose \le \delta w/((d_Y-\delta \varphi)\psi_X)}{\delta d_X d_Y g \choose \le w/((d_X-\psi_X)\psi_Y)}\enq
such that any $A \in \cG(a,g)$ for which $F'$ is a $\varphi$-approximation has a $(\psi_X, \psi_Y)$-approximating pair in $\cW.$
\end{lem}

\begin{proof}
    The proof goes very similarly to that of \cite[Lemma 5.5]{Galvin_independent} with the distinction between $\psi_X$ and $\psi_Y$ and different degrees. Fix any ordering $\prec$ on $V(\Sigma).$

\begin{quote}
    \nin \textbf{Input.} $A \in \cG(a,g)$ and its $\varphi$-approximation $F' \in \cV(a,g,\varphi)$

    \nin \textbf{Step 1.}  If $\{u \in [A]:d_{N(A)\setminus F'}(u)>\psi_X\} \ne \emptyset,$ pick the smallest $u$ (in $\prec$) in this set and update $F'$ by $F' \leftarrow F' \cup N(u).$ Repeat this until $\{u \in [A]:d_{N(A) \setminus F'}(u)>\psi_X\}=\emptyset.$ Then set $F''=F'$ and $S''=\{u \in X:d_{F''}(u) \ge d(u)-\psi_X\}$ and go to Step 2.

    \nin \textbf{Step 2.} If $\{v \in Y \setminus N(A): d_{S''}(v)>\psi_Y\} \ne \emptyset$, pick the smallest $v$ (in $\prec$) in this set and update $S''$ by $S'' \leftarrow S'' \setminus N(v).$ Repeat this until $\{v \in Y \setminus N(A):d_{S''}(v)>\psi_Y\}=\emptyset.$ Set $S=S''$ and $F=F'' \cup \{v \in Y:d_{S}(v)>\psi_Y\}$ and stop.

    \nin \textbf{Output.} $(F, S)$
\end{quote}

    We first check that the resulting $(F, S)$ is a $(\psi_X, \psi_Y)$-approximating pair for $A \in \cG(a,g)$ (see Definition~\ref{def:psi-approximation}). To see that $F\subseteq  N(A)$, note that $F'\subseteq N(A)$ at the start of the algorithm (by the definition of $\varphi$-approximation) and each element added to $F'$ in Step $1$ belongs to $N([A])=N(A)$. Moreover, $\{v \in Y: d_S(v)>\psi_Y\}\subseteq N(A)$ else Step 2 would not have terminated. To see that $S\supseteq [A]$ note that initially  $S''\supseteq [A]$, else Step 1 would not have terminated, and Step 2 deletes from $S''$ only neighbours of $Y\backslash N(A)$. To verify~\eqref{S.degree} note that at the end of Step 1, $d_{F''}(u)\geq d(u)-\psi_X$ for all $x\in S''$ by definition, and $F\supseteq F''$ and $S$ is contained in this initial set $S''$. Condition~\eqref{F.degree} is immediate from the definition of $F$.
    
    Next, We show that the number of outputs for each $F'$ is at most the right-hand side of \eqref{psi_bd}. 

    \nin \textit{Cost for Step 1.} Initially, $|N(A) \setminus F'| \le \delta w/(d_Y-\delta\varphi)$ (each vertex in $N(A) \setminus F'$ is in $N(A) \setminus N(A)^\varphi$ and so contributes at least $(\delta^{-1}d_Y-\varphi)$ edges to $\nabla(N(A), X \setminus [A])$, a set of size $\leq gd_Y-ad_X=w$). Each iteration in Step 1 removes at least $\psi_X$ vertices from $N(A) \setminus F'$ and so there can be at most $\delta w/((d_Y-\delta\varphi)\psi_X)$ iterations. The $u$'s in Step 1 are all drawn from $[A]$ and hence $N(F')$, a set of size at most $d_Yg$. So the total number of outputs for Step 1 is at most \[{d_Yg \choose \le \delta w/((d_Y-\delta\varphi)\psi_X)}.\]

    \nin \textit{Cost for Step 2.} Each $u \in S'' \setminus [A]$ contributes more than $d_X-\psi_X$ edges to $\nabla(N(A), X \setminus [A])$, so initially $|S'' \setminus [A]| \le w/(d_X-\psi_X).$
    
    Each $v$ used in Step 2 reduces this by at least $\psi_Y,$ so there are at most $w/((d_X-\psi_X)\psi_Y)$ iterations. Each $v$ is drawn from $N(S''),$ a set which is contained in the second neighbourhood of $F''$ (indeed $S''\subseteq N(F'')$ by definition) and so has size at most $\delta d_Xd_Y g.$ So the total number of outputs for Step 2 is at most
    \[{\delta d_X d_Y g \choose \le w/((d_X-\psi_X)\psi_Y)}. \qedhere \]
\end{proof}

\begin{lem}[Cost for additional run]\label{lem.heart} Let $\cW=\cW(F', \psi_X, \psi_Y)$ be as in \Cref{lem.psi}.
    For each $(F,S) \in \cW$, if $1 \le \psi_X'\leq \psi_X, \ 1\leq \psi_Y' \le \psi_Y$, then there is a family $\cW'=\cW'(F,S) \sub 2^Y \times 2^X$ with
    \beq{psi'_bd}|\cW'|\le {|S| \choose \le (g-|F|)/\psi_X'}{\delta d_X|S| \choose \le w/((d_X-\psi_X')\psi_Y')}\enq
    such that any $A \in \cG(a,g)$ for which $(F,S)\in \cW$ is a $(\psi_X, \psi_Y)$-approximating pair has a $(\psi_X', \psi_Y')$-approximating pair in $\cW'.$
\end{lem}

\begin{proof}
    For $(F,S)$ as in the statement of the lemma, we additionally run the algorithm in the proof of \Cref{lem.psi} with $\psi_X'$ and $\psi_Y'$.% Note that $\psi_X' = \psi_X \cdot O(1/K)=O(\gamma \psi_X)$, since $s-a=\gO(t)$ and $g-f\stackrel{\eqref{eq.SF}}{=}O(t)$ (\eqref{eq.SF} implies that $g-f \le (1-\psi_Y/d)^{-1}(g-s+(s-a)\psi_X/d)=O(t)$). 
\begin{quote}
    \nin \textbf{Input.} $A \in \cG(a,g)$ and its $(\psi_X, \psi_Y)$-approximating pair $(F,S)$

    \nin \textbf{Step 1'.}  If $\{u \in [A]:d_{N(A)\setminus F}(u)>\psi_X'\} \ne \emptyset,$ pick the smallest $u$ (in $\prec$) in this set and update $F$ by $F \leftarrow F \cup N(u).$ Repeat this until $\{u \in [A]:d_{N(A) \setminus F}(u)>\psi_X'\}=\emptyset.$ Then set $\hat F=F$ and $\hat S=\{u \in S:d_{\hat F}(u) \ge d(u)-\psi_X'\}$ and go to Step 2'.

    \nin \textbf{Step 2'.} If $\{v \in Y \setminus N(A): d_{\hat S}(v)>\psi_Y'\} \ne \emptyset$, pick the smallest $v$ (in $\prec$) in this set and update $\hat S$ by $\hat S \leftarrow \hat S \setminus N(v).$ Repeat this until $\{v \in Y \setminus N(A):d_{\hat S}(v)>\psi_Y'\}=\emptyset.$ Then set $\tilde S=\hat S$ and $\tilde F=\hat F \cup \{v \in Y:d_{\tilde S}(v)>\psi_Y'\}$ and stop.

    \nin \textbf{Output.} $(\tilde F, \tilde S)$
\end{quote}
It is easy to see that $(\tilde F, \tilde S)$ is a $(\psi_X', \psi_Y')$-approximating pair for $A \in \cG(a,g)$. We bound the number of outputs for each $(F,S)$. Write $f=|F|.$

    \nin \textit{Cost for Step 1'.} Each iteration in Step 1' removes at least $\psi_X'$ vertices from $N(A) \setminus F$, a set of size $g-f$, and so there can be at most ${(g-f)}/{\psi_X'}$ iterations. The $u$'s in Step 1' are all drawn from $[A]$ and hence from $S$. So the total number of outputs for Step 1' is at most \[{|S| \choose \le (g-f)/\psi_X'}.\]

    \nin \textit{Cost for Step 2'.} Each $u \in \hat S \setminus [A]$ contributes more than $d_X-\psi_X'$ edges to $\nabla(N(A), X \setminus [A])$, so initially $|\hat S \setminus [A]| \le w/(d_X-\psi_X').$ Each $v$ used in Step 2' reduces this by at least $\psi_Y',$ so there are at most $w/((d_X-\psi_X')\psi_Y')$ iterations. Each $v$ is drawn from $N(\hat S),$ a set which is contained in $N(S)$ and so has size at most $\delta d_X|S|.$ So the total number of outputs for Step 2' is at most
    \[{\delta d_X|S| \choose \le w/((d_X-\psi_X')\psi_Y')}. \qedhere\]

\end{proof}

\begin{proof}[Proof of \Cref{heart}]

Let $\gamma$ be a large enough positive constant chosen so that the following computations go through. Set
\beq{setup} \psi^X_0=d_X/\gamma \quad \text{and} \quad  \psi^Y_0=d_Y/\gamma, \enq
and let $\cW_0=\cW(F', \psi^X_0, \psi^Y_0)$ be the collection given by \Cref{lem.psi}. So
\beq{container1} \text{any $A \in \cG(a,g,F')$ has a $(\psi^X_0, \psi^Y_0)$-approximating pair in $\cW_0$.}
\enq
By \eqref{psi_bd} and the above choice of $\psi^X_0$ and $\psi^Y_0$, we have
\beq{W0.bd} |\cW_0|\leq {d_Yg \choose \le \delta \gamma w/((d_Y-\delta \varphi)d_X)}{\delta d_X d_Y g \choose \le 2\gamma w/(d_Xd_Y)}. \enq

We will iteratively run $(\psi^X_i, \psi^Y_i)$-approximations with the following parameters: given $\cW_i$, for the members $(F,S)$ of \textit{some} $\cU_i \sub \cW_i$ we will apply \Cref{lem.heart} with
\beq{psi.ini} \psi^Y_{i+1}=d_Y/\gamma \quad \text{and} \quad \psi^X_{i+1}=\psi^X_i/2=2^{-(i+1)}d_X/\gamma,\enq
until we obtain $(\psi_\kappa^X, \psi_\kappa^Y)$-approximation (see \eqref{kap} for the definition of $\kappa$). In particular, throughout our iterations, we have
\beq{i.bd} 2^i \le \sqrt {d_X} \quad \forall i\enq
and
\beq{psi.X} 1 \le \psi^X_i\leq d_X/\gamma, \ 1\leq \psi^Y_i \le d_Y/\gamma \quad \forall i\enq
(so each of the $(\psi^X_i, \psi^Y_i)$-approximations is well-defined for all $i$).

Observe that for any $i$ and any $(F,S)$, a $(\psi_i^X, \psi_i^Y)$-approximating pair for $(F,S)$ satisfies
\beq{s.t.bd} |S|-a \le 2t; \enq
indeed, by \Cref{SF_size}, \eqref{psi.X} and the fact that $d_Y\leq d_X$, we have 
\[s \le f+\frac{1}{d_X}\left[(g-f)\frac{d_Y}{\gamma}+(s-a)\frac{d_X}{\gamma}\right]\le f+(t-f+s)/\gamma,\]
which yields $s \le f+t/(\gamma-1)$. So $s-a \le f-a+t \le g-a+t=2t.$ In particular, this implies
\beq{s.2g} |S|\leq 2g. \enq

Now we define the collection $\cU_i$ for which the additional run will be performed. Let $\cU_i$ be the collection of all $(F,S)\in \cW_i$ with 
\beq{case3.F} |F| \ge g-t/2^i.\enq
Roughly speaking, $\cU_i$ consists of pairs $(F,S)$ for which an additional run of $\psi$-approximation is affordable, as is shown in the next lemma. 

\begin{claim}\label{iteration.cost}
    For each $i\geq 0$ and $(F,S) \in \cU_i$, there is a family $\cW_{i+1}(F,S) \sub 2^Y \times 2^X$ with
    \[|\cW_{i+1}(F,S)|\le \exp\left[O(t\gamma \log (gd_X^2/t)/d_X)\right]\]
    such that any $A \in \cG(a,g)$ for which $(F,S)$ is a $(\psi^X_i, \psi^Y_i)$-approximating pair has a $(\psi^X_{i+1}, \psi^Y_{i+1})$-approximating pair in $\cW_{i+1}(F,S).$
\end{claim}

\begin{proof}
We apply \Cref{lem.heart} to produce $\cW_{i+1}(F,S)$, so $|\cW_{i+1}(F,S)|$ is bounded by \eqref{psi'_bd} with $\psi'_X=\psi^X_{i+1}$ and $\psi'_Y=\psi^Y_{i+1}$:  the logarithm of the first term on the right-hand side of \eqref{psi'_bd} is then, using \eqref{binom} and \eqref{s.2g}, at most
    \[\log{2g \choose \le (g-f)/\psi^X_{i+1}}=O\left(\frac{g-f}{\psi^X_{i+1}}\log \left(\frac{g\psi^X_{i+1}}{g-f}\right)\right)\stackrel{(\ddagger)}{=}O\left(\frac{t\gamma}{d_X}\log \left(\frac{gd_X}{t}\right)\right)\]
    where $(\ddagger)$ uses the fact that 
    \beq{1overx} \text{the function $x\log(1/x)$ is monotone increasing on $(0, 1/e)$}\enq
    and 
    \[\frac{g-f}{g\psi^X_{i+1}} \stackrel{\eqref{psi.ini},\eqref{case3.F}}{\le} t/2^i\cdot 2^{i+1}\gamma/(gd_X) \stackrel{\eqref{psi.ini}}{\le} 2t\gamma/(gd_X)~(\leq 1/e).\]
Similarly, the logarithm of the second term on the right-hand side of \eqref{psi'_bd} is at most
    \[\log{2\delta d_X g \choose \le w/((d_X-\psi^X_{i+1})\psi^Y_i)}=O\left(\frac{w}{(d_X-\psi^X_{i+1})\psi^Y_i}\log \left(\frac{d_X g (d_X-\psi^X_{i+1})\psi^Y_i}{w}\right)\right)\stackrel{(*)}{=}O\left(\frac{t\gamma}{d_X}\log\left(\frac{gd_X^2}{t}\right) \right)\]
where $(*)$ follows from
\[\frac{w}{d_Xg(d_X-\psi_{i+1}^X)\psi_i^Y} \stackrel{\eqref{psi.ini}, \eqref{w.td}}{\le} \frac{td_Y\gamma}{\frac{1}{2}gd_X^2d_Y}=\frac{2t\gamma}{gd_X^2}~(<1/e),\]
recalling that 
\beq{w.td} w~(=gd_Y-ad_X)\leq td_Y. \qedhere \enq
\end{proof}

Now, set $\cW_{i+1}=\bigcup_{(F,S) \in \cU_i}\cW_{i+1}(F,S)$ and repeat the process until we obtain $\cW_{\kappa}$. Then set
\begin{align}\label{eq:FiDef}
\cF_i:=\begin{cases} \cW_i \setminus \cU_i & \text{ for $i \in [0, \kappa-1]$;}\\
\cW_\kappa & \text{ for $i=\kappa$},\end{cases}
\end{align}
and let $\mathcal{F}=\bigcup_{i \in [0, \kappa]} \cF_i$. Crucially, by construction, $\cF_i$ satisfies items (2) and (3) in \Cref{heart}. It remains to show the bound on $|\cF_i|$ in item (1).
By \Cref{iteration.cost}, for each $i\geq 0$,
    \beq{total.containers}|\cW_i|\leq |\cW_0|\exp\left[O\left(\frac{it\gamma}{d_X} \log \left(\frac{gd_X^2}{t}\right)\right)\right].\enq
Now, the conclusion follows by recalling the upper bound \eqref{W0.bd} on $|\cW_0|$.\end{proof}

\section{Proof of \Cref{ML}}\label{mlproofsec}

Our goal is to bound $\sum_{A \in \cG(a,g)} \gl^{|A|}$ by the right-hand side of \eqref{target} using \Cref{heart}. Recall the definition of $\kappa$ in \eqref{kap}. For notational simplicity, we write $A\sim_{i, \gamma}(F, S)$ if $(F,S)\in \cF_i$
is a $(d_X/(2^i\gamma),d_Y/\gamma)$-approximating pair for $A$.
By \Cref{heart}, we have
\beq{summ}\sum_{A \in \cG(a,g)} \gl^{|A|} \le \sum_{F' \in \cV(a,g,\varphi)} \sum_{i=0}^\kappa \sum_{(F,S) \in \cF_i} \sum_{A \sim_{i, \gamma} (F,S)}\gl^{|A|}.\enq
The key part of this section is to give a uniform upper bound on $\sum_{A \sim_{i, \gamma}(F,S)} \gl^{|A|}$ that is valid for all $i$ and $(F,S)$. Once we do that, we can bound the right-hand side of~\eqref{summ} by combining the bounds in \Cref{lem.varphi} and \Cref{heart}.

Recall the assumptions
\beq{t.ass} t=g-a\ge \gd' g/d_Y \text{ for a given constant $\gd'>0$},\enq
and
\beq{gl.ass} \gl>C\log^2 d_X/\sqrt{d_X}, \enq
where we get to choose $C$ given $(\gd, \gd', \gd'')$.

For convenience, we also define
\[ \bar\lam = \min\{\lam, 1\}.\]
We then have $1+\lam\geq e^{\bar\lam/2}.$

\subsection{Bounding $\sum_{A \sim_{i, \gamma}(F,S)}\gl^{|A|}$}\label{reconstruction} Suppose $i \in [0, \kappa]$ and $(F,S) \in \cF_i$ are given. Our strategy here is similar to \cite{galvin_threshold}: we apply different reconstruction strategies depending on the size of $F$ relative to $g$. First, if $i \in [0, \kappa-1]$, then we take advantage of the fact that
\beq{2.i.bd} \left(t/\sqrt {d_X} \stackrel{\eqref{i.bd}}{\le}\right)~2^{-i}t<g-f\, , \enq
where the second inequality follows from~\eqref{case3.F} and~\eqref{eq:FiDef} (the definition of $\mathcal{F}_i$). 

\begin{lem}\label{small.F}
    For each $i \in [0, \kappa-1]$, we have
    \[\sum_{A \sim_{i, \gamma}(F,S)} \gl^{|A|} \le (1+\gl)^ge^{-\bar\gl t/(3\sqrt {d_X})}.\]
\end{lem}

\begin{proof}
For the given $(F,S)$, we specify $A$ as a ($\gl$-weighted) subset of $S$, so $\sum_{A \sim_i(F,S)} \gl^{|A|} \le (1+\gl)^{|S|}$. We show that the right-hand side, $(1+\gl)^{|S|}$, is (significantly) smaller than our benchmark $(1+\gl)^g$.

Let $f=|F|$ and $s=|S|.$ By \Cref{SF_size}, \eqref{psi.ini}, and the fact that $d_Y\leq d_X$,
\[s \le f+\frac{1}{d_X}[(g-f)d_X/\qr+(s-a)2^{-i}d_X/\qr],\]
and so
\beq{same}\begin{split} g-s &\ge g-(f+(g-f)/\qr+(s-a)2^{-i}/\qr)\\
& \stackrel{\eqref{2.i.bd}}{>} (g-f)(1-1/\qr)-(s-a)(g-f)/(t\qr)\stackrel{\eqref{s.t.bd}}{\ge}(g-f)(1-1/\qr-4/\qr)\ge 0.9(g-f).\end{split}\enq
Therefore,
\beq{finalsum}\sum_{A \sim_i (F, S)} \gl^{|A|} \le (1+\gl)^{|S|} \le (1+\gl)^g(1+\gl)^{-(g-s)} \le (1+\gl)^g(1+\gl)^{-0.9(g-f)}\stackrel{\eqref{2.i.bd}}{\le }(1+\gl)^ge^{-\bar\gl t/(3\sqrt {d_X})}.\enq
\end{proof}

Notice that the proof of \Cref{small.F} crucially relies on the fact that $g-s$ is ``fairly large," which is not necessarily true for $i=\kappa$. The lemma below handles the case $i=\kappa$.

\begin{lem}\label{lem_large F} For each $(F, S) \in \mathcal{F}_\kappa$, we have
\[\sum_{A \sim_{\kappa, \gamma} (F, S)} \gl^{|A|} \le (1+\gl)^ge^{-\bar\gl t/(3\sqrt{d_X})}.\]
\end{lem}

\begin{proof} We consider two cases. If $g-f>2^{-\kappa}t~\left(\stackrel{\eqref{i.bd}}{\ge} t/\sqrt {d_X} \right)$, then we repeat the proof of \Cref{small.F}: again by \Cref{SF_size}, \eqref{psi.ini}, and the fact that $d_Y\leq d_X$,
\[s \le f+\frac{1}{d_X}[(g-f)d_X/\qr+(s-a)2^{-\kappa}d_X/\qr],\]
so by repeating the computation in \eqref{same}, we have $g-s \ge 0.9(g-f)$, and thus conclude (as in \eqref{finalsum}) that
\[\sum_{A \sim_\kappa(F,S)} \gl^{|A|} \le (1+\gl)^ge^{-\bar\gl t/(3\sqrt{d_X})}.\]

For the second case, if $g-f \le 2^{-\kappa} t \stackrel{\eqref{kap}}{<}\frac{2t}{\sqrt{d_X}}$, then we first specify $G \setminus F$ as a subset of $N(S) \setminus F$. Note that 
\[|N(S) \setminus F| \le |N(S)| \le \delta d_X|S| \stackrel{\eqref{s_loose}}{\le} \delta(d_X)^2g,\]
so the number of choices for $G \setminus F$ is at most $\displaystyle {\delta (d_X)^2g \choose \le 2t/\sqrt {d_X}}.$  Since $F \sub G$, the specification of $G \setminus F$ identifies $G$, and thus $[A]$. Then we specify $A$ as a ($\lambda$-weighted) subset of $[A]$. Therefore,
\[\sum_{A \sim_\kappa (F, S)} \gl^{|A|} \le {\delta(d_X)^2 g \choose \le 2t/\sqrt {d_X}} (1+\gl)^a =(1+\gl)^g(1+\gl)^{-t}{\delta(d_X)^2 g \choose \le 2t/\sqrt {d_X}}\stackrel{(\dagger)}{\leq}(1+\gl)^ge^{-\bar\gl t/3}<(1+\gl)^g e^{-\bar\gl t/(3\sqrt{d_X})},\]
where $(\dagger)$ uses \eqref{binom} and \eqref{t.ass} and requires for $C$ to be sufficiently large as a function of $\delta$ and $\delta'$.
\end{proof}

\subsection{Bounding the total number of containers} We first find an upper bound on $|\mathcal{V}|$ using \eqref{varphi_bd}. Set $\varphi=d_Y/(2\delta)$ and $\gamma' =\max\{1,10/\gd''\},$ noting that such $\gamma'$ satisfies the assumption of \Cref{lem.varphi} for large enough $d_X$ and $d_Y.$ In what follows, $O_{{\bar\gd}}(\cdot)$ means the implicit constant depends (only) on $\bar \gd=(\gd, \gd', \gd'')$.

We bound each term in \eqref{varphi_bd}:
\[\frac{54g\gamma'\log d_X \log (\delta d_Xd_Y)}{\varphi d_X}=O_{\bar \gd}\left(\frac{g\log^2 d_X}{d_Xd_Y}\right).\]
\[\frac{54g\log (\delta d_Xd_Y)}{(d_X)^{\gamma'm_\varphi/(\varphi d_X)}}=O_{\bar \gd}\left(\frac{g\log d_X}{(d_X)^{10}}\right).\]
\[\frac{54w\log d_Y \log (\delta d_Xd_Y)}{d_X(d_Y/\gd-\varphi)}\stackrel{\eqref{w.td}}{=}O\left(\frac{\gd td_Y\log^2 d_X}{d_Xd_Y}\right)=O_{\bar\gd}\left(\frac{t\log^2 d_X}{d_X}\right).\]
\[\log {\frac{3g\gamma'd_Y\log d_X}{\varphi d_X}\choose \le \frac{3w\gamma'\log d_X}{\varphi d_X}}\stackrel{\eqref{binom}}{=}O\left(\frac{\gd w\gamma' \log d_X}{d_Yd_X}\cdot \log\left(\frac{gd_Y}{w}\right)\right)\stackrel{\eqref{1overx},\eqref{w.td}, \eqref{t.ass}}{=}O_{\bar\gd}\left(\frac{t\log^2  d_X}{d_X}\right).\]
(Here and below, the asymptotic terms on the right-hand sides are bounded away from 0 by the assumption $t \ge c(\gd, \gd', \gd'')\cdot d_X/(\log d_X)^2$.) Therefore,
\beq{v.total} |\cV|\leq|Y| \exp\left[O_{\bar \gd} \left(\frac{g\log^2d_X}{d_Xd_Y}+\frac{t\log^2 d_X}{d_X}\right)\right].\enq

Next, we give a simpler upper bound on $\sum_{i \in [0, \kappa]}|\cF_i|$ by bounding each term in \eqref{heart.bd}:
\[\log{d_Yg \choose \le \delta \gamma w/((d_Y-\delta \varphi)d_X)}\stackrel{\eqref{binom}}{=}O\left(\frac{\gd w\gamma}{d_Xd_Y}\log\left(\frac{d_Y^2d_Xg}{\gamma\delta w}\right)\right)\stackrel{\eqref{1overx},\eqref{t.ass}}{=}O_{\bar\gd}\left(\frac{\gamma t \log d_X}{d_X}\right).\]
\[\log{\delta d_X d_Y g \choose \le 2\gamma w/(d_Xd_Y)}\stackrel{\eqref{binom}}{=}O_{\bar\gd}\left(\frac{w\gamma}{d_Xd_Y}\log\left(\frac{gd_X}{w\gamma}\right)\right)\stackrel{\eqref{1overx},\eqref{w.td}}{=}O_{\bar\gd}\left(\frac{\gamma t}{d_X}\log\left(\frac{d_X}{t\gamma}\right)\right)=O_{\bar\gd}\left(\frac{\gamma t\log d_X}{d_X}\right).\]
\[\frac{t\log d_X \log(gd_X^2/t)}{d_X}\stackrel{\eqref{1overx},\eqref{t.ass}}{=}O_{\bar\gd}\left(\frac{t\log^2 d_X}{d_X}\right).\]
Therefore, for any $F'\in \cV$,
    \beq{w.cost} \sum_{i \in [0, \kappa]}|\cF_i(a,g,F')|\le \exp\left[O_{\bar \gd}\left(\frac{\gamma t\log d_X}{d_X}\right)+O_{\bar\gd}\left(\frac{t\log^2 d_X}{d_X}\right)\right].\enq

\subsection{Summing up}\label{summing}
By putting together \eqref{v.total}, \eqref{w.cost}, \Cref{small.F}, and \Cref{lem_large F}, we upper bound the right-hand side of \eqref{summ} by the product of $|Y|(1+\gl)^ge^{-\bar\gl t/(3\sqrt{d_X})}$ and
\beq{above}\exp\left[O_{\bar \gd} \left(\frac{g\log^2d_X}{d_Xd_Y}+\frac{t\log^2 d_X}{d_X}\right)\right]\cdot \exp\left[O_{\bar \gd}\left(\frac{\gamma t\log d_X}{d_X}\right)+O_{\bar\gd}\left(\frac{t\log^2 d_X}{d_X}\right)\right].\enq
Set $C=C(\gd,\gd',\gd'')$ large enough to dominate all the implicit constants that depend on $(\gd, \gd', \gd'')$ in the above expression and the (absolute) constant $\gamma$. With this choice of $C$ and using \eqref{t.ass}, we can make \eqref{above} at most
\[\exp\left[\frac{1}{6}\cdot \frac{Ct\log^2 d_X}{d_X}\right]\stackrel{\eqref{gl.ass}}{<}\exp\left[\frac{1}{6}\cdot\frac{\bar\gl t}{\sqrt{d_X}}\right].\]
This yields the conclusion of \Cref{ML}.

\section{Proof of Theorem \ref{thm:main}}\label{pfthmmain}

The key idea to proving \Cref{thm:main} is to replace the use of Galvin's original container lemma in \cite{hypercube} by our \Cref{ML}. The rest of the proof proceeds as in \cite{hypercube}, so we introduce necessary concepts here at a rapid pace, and refer the reader to the above paper for a more detailed exposition.

Recall that we write $\cE$ and $\cO$ for the even and odd sides of $Q_d$, respectively. For an independent set $I\in\cI(Q_d)$, we say $\cE$ is the \emph{minority side} of $I$ if $|\cE\cap I|<|\cO\cap I|$ and the \emph{majority side} otherwise. If $\cE$ is the minority side of $I$, then $\cO$ is the majority side and vice versa.

For $\cD\in\{ \cE, \cO\}$ we will define a measure $\nu_\cD$ that will approximate the distribution of vertices on the minority side of a sample from $\mu_{Q_d,\lam}$. First, let $\cP=\cP_\cD$ be the family of all subsets $A$ of $\cD$ which are $2$-linked and satisfy $|[A]|\leq 2^{d-2}$. We call elements of $\cP$ \emph{polymers}. Two polymers $A_1, A_2$ are called \emph{compatible} if $A_1\cup A_2$ is not $2$-linked; we write this as $A_1\sim A_2$. Let $\Omega=\Omega_\cD$ be the collection of all sets of pairwise compatible polymers. For each polymer $A$, define the \emph{weight} of $A$ as
\[ w(A)\coloneqq \frac{\lam^{|A|}}{(1+\lam)^{|N(A)|}}.\]

Set
\[\Xi=\sum_{\Lambda\in\Omega} \prod_{A\in \Lambda} w(A)\]
and finally define $\nu_\cD$ as the probability measure on $\Omega$ given by
\[ \nu_\cD (\Lambda)=\frac{\prod_{A\in\Lambda} w(A)}{\Xi}.\]

\nin The tuple $(\cP, w, \sim)$ is an example of a \emph{polymer model} and $\Xi$ is the resulting partition function.

Now define a random independent set $\mathbf{ I}\in\cI(Q_d)$ as follows:
\begin{enumerate}
    \item choose $\cD\in\{\cE,\cO\}$ uniformly at random;\label{Step1}
    \item sample a configuration $\Lambda\in\Omega_\cD$ according to $\nu_\cD$ and include the set $\bar\Lambda:=\bigcup_{A\in\Lambda}A$ in $\mathbf{I}$;\label{Step2}
    \item for each vertex $v$ on the non-defect side that is not adjacent to  vertex in $\bar\Lambda$, include $v$ in $\mathbf{I}$ independently with probability $\frac{\lam}{1+\lam}$.\label{Step3}
\end{enumerate}
Let $\hat{\mu}_\lam$ denote the law of the random independent set $\mathbf{I}\in\cI(Q_d)$. We refer to the side $\cD\in\{\cE,\cO\}$ selected at Step 1 as the \emph{defect side} of $\mathbf{I}$. We will see (\Cref{lem17} below) that with high probability the defect side of $\mathbf{I}$ is the minority side. 

We borrow one more concept from the theory of  polymer models. Once again fix $\cD\in\{\cE,\cO\}$. A \emph{cluster} is an ordered tuple $\Gamma$ of polymers in $\cP_\cD$ such that the incompatibility graph $H(\Gamma)$ (defined on the vertex set $\Gamma$, with $\{A_1,A_2\}$ an edge if and only if $A_1\not\sim A_2$) is connected. The \emph{size} of a cluster $\Gamma$ is $\|\Gamma\|\coloneqq \sum_{A\in\Gamma} |A|$. We write $\cC=\cC_\cD$ for the set of all clusters, and $\cC_k$ for the set of all clusters of size $k$ ($k\geq 1$).

The \emph{Ursell function} of a graph $G=(V,E)$ is
\[ \phi(G)\coloneqq \frac{1}{|V|!} \sum_{\substack{A\subseteq E \\ (V,A)\text{ connected}}} (-1)^{|A|}.  \]
For a cluster $\Gamma\in\cC$, define
\[ w(\Gamma)=\phi(H(\Gamma))\prod_{S\in\Gamma} w(S)\]
and for $k\geq 1$  set
\[ L_k\coloneqq \sum_{\Gamma\in\cC_k} w(\Gamma)\]
and
\[ T_k\coloneqq \sum_{i=1}^{k-1} L_i.\]

The \emph{cluster expansion} is the formal power series
\[ \log \Xi = \sum_{\Gamma\in \cC} w(\Gamma).\]

\nin We will make use of the following criterion for convergence of the cluster expansion.
\begin{thm}[Koteck\'y and Preiss {\cite{kp}}]
    Let $f, g:\cP\rightarrow[0, \infty)$ be two functions. If, for all polymers $A\in\cP$, we have
    \beq{kpcond} \sum_{A'\not\sim A} w(A')e^{f(A')+g(A')} \leq f(A)\enq
    then the cluster expansion converges absolutely. Moreover, writing $\Gamma\not\sim A$ if there exists some $A'\in\Gamma$ with $A'\not\sim A$ and setting $g(\Gamma)=\sum_{A\in\Gamma} g(A)$, for all polymers $A$ we have
    \beq{kpbd} \sum_{\Gamma\not\sim A} |w(\Gamma)|e^{g(\Gamma)}\leq f(A). \enq
\end{thm}

We will also require the following result about isoperimetry in the hypercube.

\begin{lem}[see e.g. \cite{galvin.hom,sapokorcube}]

    Let $A\subseteq \cE$ (or $A\subseteq \cO$). Then:
    \begin{enumerate}
        \item if $|A|\leq d/10$, then $|N(A)|\geq d|A|-|A|^2$;
        \item if $|A|\leq d^4$, then $|N(A)|\geq d|A|/10$;
        \item if $|A|\leq 2^{d-2}$, then $|N(A)|\geq \left( 1+\frac{1}{2\sqrt{d}}\right)|A|$.
    \end{enumerate}
\end{lem}

\nin In particular, the above lemma ensures that $Q_d$ satisfies the expansion properties required to apply \Cref{ML}.

The key step that enables one to derive \Cref{thm:main} is the following result, which states that the measure $\hat{\mu}_\lam$ is a very close approximation to the hard-core measure $\mu_{Q_d,\lam}$.

\begin{thm}[\textit{cf.} {\cite[Lemma 14]{hypercube}}]\label{clusterapprox}
    There exists a sufficiently large absolute constant $C>0$ such that if $\lam\geq C\log^2 d/d^{1/2}$ and $\lam$ is bounded as $d\rightarrow\infty$, then we have
    \[ \left| \log Z_{Q_d}(\lam) -\log \left[ 2(1+\lam)^{2^{d-1}}\Xi\right]\right| = O\left(\exp(-2^d/d^4)\right)\]
  and
    \beq{tvbd} \left\|\mu_{Q_d,\lam}-\hat{\mu}_\lam \right\|_{TV}=O\left(\exp(-2^d/d^4)\right). \enq
\end{thm}

Just like \cite[Lemma 14]{hypercube}, the proof of the above result relies on establishing convergence of the cluster expansion for $\log \Xi$ via the Koteck\'y-Preiss criterion. Crucially, we use \Cref{ML} in order to do this, replacing the use of Galvin's original container lemma as it appears in \cite{hypercube}. For brevity, we will henceforth assume without mention that $d$ is sufficiently large for all the bounds we state to hold, and $\lam$ satisfies $\lam\geq C\log^2 d/d^{1/2}$.

\begin{lem}[\textit{cf.} {\cite[Lemma 15]{hypercube}}]\label{lem15}
    For integers $d, k\geq 1$, let
    \[\gamma(d, k) = \begin{cases*}
        \log (1+\lam)(dk-3k^2) - 7k\log d& if $k\leq \frac{d}{10}$, \\
        \frac{d\log(1+\lam)k}{20}& if $\frac{d}{10}<k\leq d^4$,\\
        \frac{k}{d^{3/2}}& if $k>d^4$.
    \end{cases*} \]
    Then,
    \beq{kpbdcube} \sum_{\substack{\Gamma\in\cC \\ \|\Gamma\|\geq k}} |w(\Gamma)| \leq d^{-3/2}2^{d-1}e^{-\gamma(d, k)}. \enq
    In particular, for $k$ fixed and $d$ sufficiently large,
    \beq{kptailscube} |T_k - \log \Xi| \leq d^{7k-3/2}2^d(1+\lam)^{-dk+3k^2}. \enq
\end{lem}
\begin{proof}
    Let $g:\cP\rightarrow[0, \infty)$ be defined by $g(A)=\gamma(d, |A|)$ and $f:\cP\rightarrow[0, \infty)$ by $f(A)=|A|/d^{3/2}$. We show that the Koteck\'y-Preiss condition \eqref{kpcond} holds, i.e. that for every $A\in\cP$,
    \beq{kpcube1} \sum_{A'\not\sim A} w(A')e^{d^{-3/2}|A'| + g(A')}\leq |A|/d^{3/2}. \enq
    To this end, we first show that for every $v\in\cE$,
    \[ \sum_{A \ni v} w(A)e^{d^{-3/2}|A|+g(A)} \leq \frac{1}{d^{7/2}},\]
    and this suffices since $A'\not\sim A$ if and only if $A'\ni v$ for some $v\in N^2(A)$ and $|N^2(A)|\leq d^2|A|$. We break up the sum according to the different cases of $\gamma(d, k)$.
    In \cite[Lemma 15]{hypercube} it is proven that
    \[ \sum_{\substack{A\ni v \\ |A|\leq\frac{d}{10}}} w(A)e^{f(A)+g(A)}\leq \frac{1}{3d^{7/2}}\]
    for sufficiently large $d$, and similarly
    \[ \sum_{\substack{A\ni v \\ d/10< |A|\leq d^4}} w(A)e^{f(A)+g(A)} \leq \frac{1}{3d^{7/2}}.\]

\nin    Indeed, the proofs of the above facts only rely on the weaker assumption that $\lam=\tilde{\Omega}(1/d)$.

    Now, for the case of $d^4<|A|\leq 2^{d-2}$, we have that $|N(A)|\geq |A|(1+1/(2\sqrt{d}))$, and so
    \begin{align*} \sum_{\substack{A\ni v\\ d^4<|A|\leq 2^{d-2}}} w(A)e^{f(A)+g(A)} &\leq \sum_{\substack{ d^4<a\leq 2^{d-2} \\ (1+1/(2\sqrt{d}))a\leq g\leq 2^{d-1}}} \sum_{\substack{A\ni v \\ |[A]|=a, |N(A)|=g}} \frac{\lam^{|A|}}{(1+\lam)^g} e^{2|A|d^{-3/2}} \\ &\leq \sum_{\substack{ d^4<a\leq 2^{d-2} \\ (1+1/(2\sqrt{d}))a\leq g\leq 2^{d-1}}} e^{2ad^{-3/2}} \sum_{\substack{A\ni v \\ |[A]|=a, |N(A)|=g}} \frac{\lam^{|A|}}{(1+\lam)^g} \\ &\leq \sum_{\substack{a>d^4 \\ g\geq (1+1/(2\sqrt{d}))a}} e^{2ad^{-3/2}} 2^d \exp\left( -\frac{(g-a)\log^2 d}{6d}\right)
    \end{align*}
    where the last inequality follows from \Cref{ML}. In the sum, we have $(g-a)\geq a/(2\sqrt{d})$ and $a>d^4$, and so
    \[ \frac{2a}{d^{3/2}} + d - \frac{(g-a)\log^2d}{6d} \leq -ad^{-3/2}\]
    for large enough $d$, and so
    \begin{align*} \sum_{\substack{A\ni v \\ d^4<|A|\leq 2^{d-2}}} w(A)e^{f(A)+g(A)} &\leq \sum_{\substack{d^4<a\leq 2^{d-2} \\ (1+1/(2\sqrt{d}))a\leq b\leq 2^{d-1}}} \exp(-ad^{-3/2}) \\ &\leq 2^d \sum_{a>d^4} \exp(-ad^{-3/2}) \leq \frac{1}{3d^{7/2}}
    \end{align*}
    for $d$ large enough. Putting the three bounds together gives \eqref{kpcube1}.

    We may now conclude like in \cite{hypercube}: applying \eqref{kpbd} for the polymer $A$ consisting of the single vertex $v$ we obtain
    \[ \sum_{\Gamma\in\cC , \Gamma\not\sim v} |w(\Gamma)|e^{g(\Gamma)} \leq d^{-3/2}.\]
    Summing over all $v$ gives
    \beq{kpcube2} \sum_{\Gamma\in\cC} |w(\Gamma)|e^{g(\Gamma)} \leq 2^{d-1}d^{-3/2}. \enq
    Since $\gamma(d, k)/k$ is non-increasing in $k$, we have, recalling that $g(\Gamma)=\sum_{A\in\Gamma} g(A)$ and $\|\Gamma\|=\sum_{A\in\Gamma} |A|$,
    \[ g(\Gamma)=\sum_{A\in\Gamma} \gamma(d, |A|)\geq \sum_{A\in\Gamma} \frac{|A|}{\|\Gamma\|} \gamma(d, \|\Gamma\|) = \gamma(d, \|\Gamma\|). \]
    Keeping only terms corresponding to clusters of size at least $k$ in inequality \eqref{kpcube2}, we have
    \[ \sum_{\substack{\Gamma\in\cC \\ \|\Gamma\| \geq k}} |w(\Gamma)| \leq d^{-3/2}2^{d-1}e^{-\gamma(d, k)}\]
    as desired.
\end{proof}

To summarize, we have extended the statement of \cite[Lemma 15]{hypercube} to the regime of $\lam=\Omega(\log^2 d/d^{1/2})$. The proof of \Cref{clusterapprox} now follows identically to the derivation of Lemma 14 in \cite{hypercube}, so we omit the details here. In particular, the following intermediate result of \cite{hypercube} holds in our extended regime for $\lam$.

\begin{lem}[\textit{cf.} {\cite[Lemma 17]{hypercube}}] \label{lem17}
    With probability at least $1-O(\exp(-2^d/d^4))$ over the random independent set $\mathbf{I}$ drawn from $\hat{\mu}_\lam$, the minority side of $\mathbf{I}$ is the defect side.
\end{lem}

The final ingredient we need for the proof of \Cref{thm:main} is the following result from \cite{struct} (or, more precisely, the statement of this result for our extended range of $\lam$, which follows from \Cref{lem15} identically to the original).
\begin{lem}[\textit{cf.} {\cite[Lemma 11]{struct}}] \label{lem11}
    For any $\ell\geq 1$ fixed,
    \[ \sum_{\substack{\Gamma\in\cC \\ \|\Gamma\|\geq 2}} w(\Gamma)\|\Gamma\|^{\ell} =O\left(\frac{2^d d^2}{(1+\lam)^{2d}} \right).\]
\end{lem}

\begin{proof}[Proof of \Cref{thm:main}]
    First note that by \eqref{tvbd}, we may instead prove the result for $\hat\mu_\lam$. We let $\mathbf{I}$ denote an independent set sampled according to $\hat\mu_\lam$. Let $\cD$ denote the defect side of $\mathbf{I}$ chosen at Step~\ref{Step1} in the definition of $\hat\mu_\lam$. Without loss of generality we assume that $\cD=\cE$.
By \Cref{lem17}, we see that \eqref{expcube} reduces to estimating the typical size of $|\mathbf{I}\cap \cE|$. 

Let $\mathbf{\Lambda}\in \Omega_\cE$ denote the  configuration chosen at Step~\ref{Step2} in the definition of $\hat\mu_\lam$ and note that $|\mathbf{I}\cap \cE|= |\mathbf{\Lambda}|$. Recall that we set $\bar\Lambda:=\bigcup_{A\in\Lambda}A$. Then, as noted in \cite{struct}, we have
\begin{align} \label{eq:expclust}
\mathbb{E} |\mathbf{\mathbf{\bar \Lambda}}| = \sum_{\Gamma\in \cC} w(\Gamma)\|\Gamma\| 
\end{align}
and
\[ \text{Var} |\mathbf{\bar \Lambda}| = \sum_{\Gamma\in\cC} w(\Gamma) \|\Gamma\|^2\, . \]

Consider first the contribution to the right-hand side of~\eqref{eq:expclust} coming from clusters of size 1. There are precisely $2^{d-1}$ of them, and each of them has weight $\lam/(1+\lam)^d$.    On the other hand, by \Cref{lem11} we know that
\[ \sum_{\substack{\Gamma\in\cC \\ \|\Gamma\|\geq 2}} w(\Gamma)\|\Gamma\| =O\left(\frac{2^d d^2}{(1+\lam)^{2d}} \right),\]

\nin    we hence see that $\mathbb{E} |\mathbf{\bar \Lambda}| = (1+o(1))\frac{\lam 2^{d-1}}{(1+\lam)^d}$. Similarly we have $\text{Var} |\mathbf{\bar \Lambda}| = (1+o(1))\frac{\lam 2^{d-1}}{(1+\lam)^d}$.
      We therefore obtain \eqref{expcube} by Chebyshev's inequality.
    
    For the proof of \eqref{expbig} we note that by Step~\ref{Step3} in the definition of $\hat\mu_\lam$,  $|\mathbf{I}\cap\cO|\sim\text{Bin}\left(2^{d-1}-O\left(\frac{d\lam 2^{d-1}}{(1+\lam)^d}\right), \frac{\lam}{1+\lam}\right)$ and the result follows again by Chebyshev's inequality and noting that $|\mathbf{I}|=|\mathbf{I} \cap \cE|+|\mathbf{I}\cap \cO|$.
\end{proof}

\section{Proof of Theorem~\ref{algos}}\label{pfthmalgos}

The proof of \cite[Theorem 2]{jenssen2023approximately} also rests on the use of Galvin's container lemma. We wish to apply \Cref{ML} in order to prove \Cref{algos}, but $\alpha$-expanders might not satisfy the condition on $m_\varphi$ required for \Cref{ML} to hold. We outline how to circumvent this issue. The rest of the proof then goes through exactly as in~\cite{jenssen2023approximately} and so
we keep this section brief in order to avoid excessive repetition. 

We begin by recalling that \cite[Theorem 2]{jenssen2023approximately} establishes Theorem~\ref{algos} in the restricted range $\lam=\tilde \Omega(d^{-1/4})$. By choosing $C$ sufficiently large in the statement of Theorem~\ref{algos} we may therefore assume that $d$ is sufficiently large to support our assertions. 

We require the following lemma from \cite{jenssen2023approximately}, which enables us to construct $\varphi$-approximations as `cheaply' in this new setting as \Cref{lem.varphi} did before. For a bipartite graph $\Sigma$, recall the definition of $\cG(a,g)=\cG(a,g,\Sigma)$ from~\eqref{eq:GagDef}.

\begin{lem}[{\cite[Lemma 7]{jenssen2023approximately}}]\label{lem.ess}
	Let $\alpha>0$ be a constant and let $\Sigma=X\cup Y$ be a bipartite $d$-regular $\alpha$-expander. Let $a, g\in \mathbb{N}$. Then there is a family $\cV=\cV(a, g)\subseteq 2^Y$ with
\beq{eq.ess} |\cV|\leq |Y|\exp \left(O\left(\frac{g\log^2 d}{d}\right)\right) \enq
such that each $A\in \cG(a, g)$ has a $d/2$-approximation in $\cV$.
\end{lem}

Using this in place of \Cref{lem.varphi}, we obtain the following container lemma.

\begin{lem}\label{lem.algos}
    For every constant $\alpha>0$ there exists a constant $C>0$ such that for $\lam\geq \frac{C\log^2 d}{d^{1/2}}$ the following holds.
    Let $\Sigma=X\cup Y$ be a bipartite $d$-regular $\alpha$-expander. Then, for any $1\leq a\leq |X|/2$ and $g\geq 1$, we have:
    \[ \sum_{A\in \cG(a, g)} \lam^{|A|} \leq |Y|(1+\lam)^g \exp(-\Omega((g-a)/d)). \]
\end{lem}
\begin{proof}
Note that the statement is trivially true for $g<d$, so we henceforth consider the case of $g\geq d$. Now, under the assumptions on $\lam$, $\Sigma$ and $a$, we have $t\geq \alpha g/(1+\alpha) =\Omega(d/\log^2 d)$. Moreover, $\alpha$-expansion implies condition \eqref{t.ass} with $\delta'=\alpha/(1+\alpha)$.

We cannot apply \Cref{ML}, as we are missing the hypothesis that $m_{d/2}=\Omega(d^2)$. However, this hypothesis was only needed in the construction of the family $\cV$ of $\varphi$-approximations (i.e. in applying \Cref{lem.varphi}) and we replace this by \Cref{lem.ess}.  So let $\cV(a,g)$ be the family of $d/2$-approximations from \Cref{lem.ess}.  We now follow the rest of the proof of \Cref{ML}, i.e. we apply \Cref{heart} (with $\delta=1$, since $\Sigma$ is regular) to obtain a family $\cF(F')$ satisfying conditions (1)-(3) in the statement of \Cref{heart} for each $F'\in\cV(a,g)$. In place of~\eqref{summ} we now instead have
\beq{summ'}\sum_{A \in \cG(a,g)} \gl^{|A|} \le \sum_{F' \in \cV(a,g)} \sum_{i=0}^\kappa \sum_{(F,S) \in \cF_i} \sum_{A \sim_{i, \gamma} (F,S)}\gl^{|A|}.\enq

We may now proceed to a reconstruction step just like for the proof of \Cref{ML}. Indeed, the computations in \Cref{reconstruction} go through unchanged, and so Lemmas \ref{small.F} and \ref{lem_large F} are still valid with $\delta'=\alpha/(1+\alpha)$. Similarly, the bound \eqref{w.cost} holds identically to before, and now instead of \eqref{v.total} we bound the size of the initial $d/2$-approximations by \eqref{eq.ess}.

Set $\alpha_{min}=\min\{1/3, \alpha\}$. By putting together \eqref{w.cost}, \eqref{eq.ess}, \Cref{small.F}, and \Cref{lem_large F}, we upper bound the right-hand side of \eqref{summ'} by the product of \[|Y|(1+\gl)^ge^{-\alpha_{min} \bar\gl t/\sqrt{d}}\] and
\beq{above'}\exp\left[O_{\delta'}\left(\frac{g\log^2 d}{d}\right)\right]\cdot \exp\left[O_{\delta'}\left(\frac{\gamma t\log d}{d}\right)+O\left(\frac{t\log^2 d}{d}\right)\right].\enq
Set $C$ large enough to dominate all the implicit constants depending on $\delta'$ in the above expression, as well as $1/\alpha_{min}$ and the (absolute) constant $\gamma$. With this choice of $C$ and recalling that $t\geq \alpha g/(1+\alpha)$ and $\bar\lam\geq C\log^2 d/d^{1/2}$, we can make \eqref{above'} at most
\[\exp\left[\frac{\alpha_{min}}{2}\cdot \frac{Ct\log^2 d}{d}\right]<\exp\left[\frac{\alpha_{min}}{2}\cdot\frac{\bar\gl t}{\sqrt{d}}\right].\]
This yields the conclusion of \Cref{lem.algos}.
\end{proof}

Using \Cref{lem.algos} in the derivation of \cite[Theorem 2]{jenssen2023approximately} in place of \cite[Lemma 20]{jenssen2023approximately} and noting that the rest of the proof goes through unchanged, we obtain \Cref{algos}.

\section{Concluding remarks}\label{conclusion}
We once again highlight  \Cref{conj}, which states that a bound of the type we obtained in \Cref{ML} should hold in the larger range $\lam=\Tilde{\Omega}(1/d)$. We note however that our present methods seem to hit a natural barrier at $\lam=\Tilde{\Omega}(1/d^{1/2})$, as even the initial construction of $\varphi$-approximations is too `costly' for values of $\lam$ smaller than this. It seems that any further progress towards \Cref{conj} would require significant new ideas.

We end with a discussion on phase transitions for the hard-core model on $\Z^d$, one of the original motivations for the results mentioned in the introduction. 

Let $G$ be the nearest neighbour graph on $\mathbb Z^d$. We say a measure $\mu$ on $\mathcal{I}(G)$ is a \emph{hard-core measure at activity $\lam$}, and write $\mu\in \text{hc}(\lam)$, if it satisfies the following property: let $\mathbf I$ be sampled according to $\mu$ and let $W\subseteq \mathbb{Z}^d$ be finite; then the conditional distribution of $\mathbf I \cap W$ given $\mathbf I \cap W^c$ is $\mu$-almost surely the same as $\mu_{H,\lam}$ where $H$ is the subgraph of $G$ induced by those vertices of $W$ not adjacent to any vertex of $I \cap W^c$. The general theory of Gibbs measures -- of which the hard-core measure is a particular instance -- guarantees the existence of at least one such $\mu$ (see e.g. \cite{georgii2011gibbs}). The question of whether $\mu$ is unique depends on $\lambda$ and is a central consideration in statistical physics. If $|\text{hc}(\lam)|>1$, then we say that the hard-core model exhibits \emph{phase coexistence} at $\lambda$ and the model is said to have undergone a \emph{phase transition}. A folklore conjecture is that there exists a critical $\lam_c(d)$ such that $|\text{hc}(\lam)|=1$ for $\lam<\lam_c(d)$ and $|\text{hc}(\lam)|>1$ if $\lam>\lam_c(d)$, and this remains an important open problem. Let $\lam(d)=\sup\{\lam: |\text{hc}(\lam)|=1 \}$. It is well-known that $\lam(d)=\Omega(1/d)$ (see e.g.~\cite{vdBS}). A seminal result of Dobrushin~\cite{dobrushinhardcore} shows that $\lam(d)<\infty$ (the proof shows that $\lam(d)<C^d$ for some $C>1$). It is commonly believed that $\lam(d)=\tilde O(1/d)$, but even showing that $\lam(d)=o(1)$ eluded efforts until a breakthrough of Galvin and Kahn~\cite{GalK04} established the bound $\lam(d)=\tilde O(d^{-1/4})$ with an approach based on the method of graph containers. This bound was later improved by Peled and Samotij \cite{peledsamotij}, who showed that $\lam(d)=\tilde O(d^{-1/3})$. It would be interesting to investigate whether the ideas of the current paper can be used to improve this bound further. 

\section*{Acknowledgements}
We thank the anonymous referees for their careful reading of this paper and their helpful comments. AM would like to thank Victor Souza for helpful discussions during the preparation of this manuscript.  MJ is supported by a UK Research and Innovation Future Leaders Fellowship MR/W007320/2. JP is supported by NSF grant DMS-2324978 and a Sloan Fellowship.

\bibliographystyle{plain}
\bibliography{bibliography}

\appendix
\section{Proof of Lemma \ref{lem.varphi}}\label{app:varphi}

For integer $s\geq 1$, we say a set $A\subseteq X$ is $s$-linked if $\Sigma^s[A]$ is connected where $\Sigma^s$ is the $s^{\text{th}}$ power of $\Sigma$ (i.e. $V(\Sigma^s)=V(\Sigma)$ and two vertices $x,y$ are adjacent in $\Sigma^s$ iff their distance in $\Sigma$ is at most $s$). We note that since $d(x)\leq \delta d_X$ for all $x\in X$ and $d(y)\leq d_Y$ for all $y\in Y$, the maximum degree, $\Delta(\Sigma^s)$, of $\Sigma^s$ is at most $(\gd d_X)^{\lceil s/2 \rceil}(d_Y)^{\lfloor s/2 \rfloor}$.

We require the following two standard results from the literature.

\begin{lem}[\cite{knuth}, p. 396, ex. 11]\label{linked}
    The number of $s$-linked subsets of $X$ of size $\ell$ that contain a fixed vertex is at most $(e \Delta(\Sigma^s))^{\ell-1} \le $ $\exp(s\ell\log(\delta d_Xd_Y))$.
\end{lem}

For a bipartite graph with bipartition $P \cup Q$ we say that $Q'\sub Q$ \emph{covers} $P$ if each $v\in P$ has a neighbour in $Q'$.

\begin{lem}[Lov\'asz \cite{Lov75}, Stein \cite{Ste74}]\label{lem.LS} Let $G$ be a bipartite graph with bipartition $P \cup Q$, where $d(u)\ge a$ for each $u \in P$ and $d(v)\le b$ for each $v \in Q$. Then there exists some $Q' \sub Q$ that covers $P$ and satisfies
\[|Q'|\le\frac{|Q|}{a}\cdot(1+\log b).\]
\end{lem}

Now fix an arbitrary set $A\in \mathcal{G}(a,g)$ and set $p=\frac{\gamma'\log d_X}{\varphi d_X}$.

\begin{claim}[{\textit{cf.} \cite[Claim 1]{BalGLW21}}]
    There exists a set $T_0\subseteq N(A)$ such that
    \beq{claim1} |T_0|\leq 3gp, \enq
    \beq{claim2} e(T_0, X\setminus [A]) \leq 3wp, \enq
    and
    \beq{claim3} |N(A)^{\varphi}\setminus N(N_{[A]}(T_0))| \leq 3g\exp(-pm_{\varphi}).\enq
\end{claim}
\begin{proof}
    Construct a random subset $\mathbf{\tilde{T}}\subseteq N(A)$ by taking each $y\in N(A)$ independently with probability $p$. We then have $\mathbb{E}(|\mathbf{\tilde{T}}|)=gp$ and $\mathbb{E}(e(\mathbf{\tilde{T}}, X\setminus[A]))=e(N(A), X\setminus[A])p \leq (gd_Y-ad_X)p=wp.$

    Now, by the definition of $N(A)^\varphi$ and $m_\varphi$, for every $y\in N(A)^\varphi$ we have $|N(N_{[A]}(y))|\geq m_\varphi$, and therefore
    \begin{align*} \mathbb{E}(N(A)^\varphi \setminus N(N_{[A]}(\mathbf{\tilde{T}})))&= \sum_{y\in N(A)^\varphi} \mathbb{P}(y\not\in N(N_{[A]}(\mathbf{\tilde{T}})))=\sum_{y\in N(A)^\varphi} \mathbb{P}(\mathbf{\tilde{T}}\cap N(N_{[A]}(y)) =\emptyset) \\ &\leq g(1-p)^{m_\varphi} \leq g\exp(-pm_\varphi).
    \end{align*}
    Applying Markov's inequality, we can find a set $T_0\subseteq N(A)$ satisfying the desired conditions.
\end{proof}

Define
\[T_0'=N(A)^\varphi \setminus N(N_{[A]}(T_0)), \quad L=T_0' \cup N(N_{[A]}(T_0)), \quad \Omega=\nabla(T_0, X \setminus [A]).\]
Let $T_1 \sub N(A) \setminus L$ be a minimal set that covers $[A] \setminus N(L)$ in the graph $\Sigma$ induced on $[A] \setminus N(L) \cup N(A) \setminus L$. Let $F':=L \cup T_1$. Then $F'$ is a $\varphi$-approximation of $A$. Note that
\[|N(A) \setminus L|(\delta^{-1}d_Y-\varphi) \le e(N(A), X \setminus [A])\leq w,\]
and thus by \Cref{lem.LS} we have
\beq{claim4} |T_1|\leq \frac{|N(A)\setminus L|}{d_X}(1+\log d_Y) \leq \frac{3w \log d_Y}{d_X(\delta^{-1}d_Y-\varphi)}.\enq

Let $T\coloneqq T_0\cup T_0' \cup T_1$. By \eqref{claim1}, \eqref{claim3} and \eqref{claim4} we have
\beq{claim5} |T|\leq 3gp + 3g\exp(-pm_\varphi) + \frac{3w \log d_Y}{d_X(\delta^{-1}d_Y-\varphi)} \eqqcolon \tau.\enq

\begin{claim}[{\cite[Claim 2]{BalGLW21}}]
    $T$ is an $8$-linked subset of $Y$.
\end{claim}

Now, since $N(N_{[A]}(T_0))$ is determined by $T_0$ and $\Omega$, it follows that $F'$ is determined by $T_0$, $\Omega$, $T_0'$, and $T_1$. Let $\mathcal{V}$ be the collection of all sets $F'$ that we obtain in such a way from some tuple $(T_0, T'_0, T_1, \Omega)$. having started from any choice of $A\in \mathcal{G}(a,g)$. By \Cref{linked}, the number of possible choices for $T$ is at most
\[ |Y|\sum_{\ell\leq \tau} \exp(8\ell \log (\delta d_Xd_Y)) \leq |Y|\exp(16\tau\log (\delta d_Xd_Y)). \]

For a fixed $T$, the number of choices for $T_0$ and $T'_0$ is at most $2^{\tau}$ each, and this uniquely determines $T_1$. Finally, by \eqref{claim2}, for a fixed $T_0$ the number of choices for $\Omega$ is at most $\binom{3gpd_Y}{\leq 3wp}.$ Putting everything together, we have
\begin{align*}
    |\mathcal{V}| \leq |Y| \exp\big(16\tau\log (\delta d_Xd_Y)\big) \cdot 4^{\tau} \binom{3gpd_Y}{\leq 3wp} \leq |Y| \exp\big(18\tau\log (\delta d_Xd_Y)\big) \binom{3gpd_Y}{\leq 3wp} \\ \leq |Y| \exp\left( \frac{54g\gamma'\log d_X \log (\delta d_Xd_Y)}{\varphi d_X} + \frac{54g\log (\delta d_Xd_Y)}{(d_X)^{\gamma'm_\varphi/(\varphi d_X)}}+\frac{54w\log d_Y \log (\delta d_Xd_Y)}{d_X(\delta^{-1}d_Y-\varphi)} \right) \displaystyle\binom{\frac{3g\gamma'd_Y \log d_X}{\varphi d_X}}{\leq\frac{ 3w\gamma' \log d_X}{\varphi d_X}}.
\end{align*}
\newline

\end{document}